\documentclass{article}
\usepackage{oldlfont}
\usepackage{enumerate}
\usepackage{amsmath}
\usepackage{amssymb}
\usepackage{amsthm}
\usepackage{mathrsfs}
\usepackage{amscd}
\usepackage{exscale}
\usepackage{latexsym}
\usepackage[all]{xy}
\usepackage{hyperref}
\usepackage{fancyhdr}
\usepackage{layout}
\usepackage{color}
\usepackage{graphicx}
\usepackage{mathtools}
\usepackage{amsfonts}
\usepackage{calc}
\usepackage{tikz}
\usepackage{mathtools}
\usepackage[OT2,T1]{fontenc}

\DeclareMathAlphabet{\mathpzc}{OT1}{pzc}{m}{it}
\DeclareSymbolFont{cyrletters}{OT2}{wncyr}{m}{n}
\DeclareMathSymbol{\Sha}{\mathalpha}{cyrletters}{"58}

\begin{document}

\baselineskip=17pt

\pagestyle{headings}

\numberwithin{equation}{section}

\makeatletter                                                           

\def\section{\@startsection {section}{1}{\z@}{-5.5ex plus -.5ex         
minus -.2ex}{1ex plus .2ex}{\large \bf}}                                 


\pagestyle{fancy}
\renewcommand{\sectionmark}[1]{\markboth{ #1}{ #1}}
\renewcommand{\subsectionmark}[1]{\markright{ #1}}
\fancyhf{} 
\fancyhead[LE,RO]{\slshape\thepage}
\fancyhead[LO]{\slshape\rightmark}
\fancyhead[RE]{\slshape\leftmark}

\addtolength{\headheight}{0.5pt} 
\renewcommand{\headrulewidth}{0pt} 

\newtheorem{thm}{Theorem}[section]
\newtheorem{mainthm}[thm]{Main Theorem}
\newtheorem*{T}{Theorem 1'}

\newcommand{\ZZ}{{\mathbb Z}}
\newcommand{\GG}{{\mathbb G}}
\newcommand{\Z}{{\mathbb Z}}
\newcommand{\RR}{{\mathbb R}}
\newcommand{\NN}{{\mathbb N}}
\newcommand{\GF}{{\rm GF}}
\newcommand{\QQ}{{\mathbb Q}}
\newcommand{\CC}{{\mathbb C}}
\newcommand{\FF}{{\mathbb F}}

\newtheorem{lem}[thm]{Lemma}
\newtheorem{cor}[thm]{Corollary}
\newtheorem{pro}[thm]{Proposition}
\newtheorem*{proposi}{Proposition \ref{pro:pro63}}
\newtheorem*{thm_notag}{Theorem}
\newtheorem{problem}{Problem}

\newtheorem{proprieta}[thm]{Property}
\newcommand{\pf}{\noindent \textbf{Proof.} \ }
\newcommand{\eop}{${\Box}$  \relax}
\newtheorem{num}{equation}{}

\theoremstyle{definition}
\newtheorem{rem}[thm]{Remark}
\newtheorem{rems}[thm]{Remarks}
\newtheorem{D}[thm]{Definition}
\newtheorem{Not}{Notation}

\newtheorem{Def}{Definition}

\newcommand{\nsplit}{\cdot}
\newcommand{\GGG}{{\mathfrak g}}
\newcommand{\GL}{{\rm GL}}
\newcommand{\SL}{{\rm SL}}
\newcommand{\SP}{{\rm Sp}}
\newcommand{\LL}{{\rm L}}
\newcommand{\Ker}{{\rm Ker}}
\newcommand{\la}{\langle}
\newcommand{\ra}{\rangle}
\newcommand{\PSp}{{\rm PSp}}
\newcommand{\Uni}{{\rm U}}
\newcommand{\GU}{{\rm GU}}
\newcommand{\GO}{{\rm GO}}
\newcommand{\Aut}{{\rm Aut}}
\newcommand{\Alt}{{\rm Alt}}
\newcommand{\Sym}{{\rm Sym}}

\newcommand{\isom}{{\cong}}
\newcommand{\z}{{\zeta}}
\newcommand{\Gal}{{\rm Gal}}
\newcommand{\SO}{{\rm SO}}
\newcommand{\SU}{{\rm SU}}
\newcommand{\PGL}{{\rm PGL}}
\newcommand{\PSL}{{\rm PSL}}
\newcommand{\loc}{{\rm loc}}
\newcommand{\Sp}{{\rm Sp}}
\newcommand{\PUni}{{\rm PU}}
\newcommand{\Id}{{\rm Id}}
\newcommand{\s}{{\sigma}}

\newcommand{\F}{{\mathbb F}}
\renewcommand{\O}{{\cal O}}
\newcommand{\Q}{{\mathbb Q}}
\newcommand{\R}{{\mathbb R}}
\newcommand{\N}{{\mathbb N}}
\newcommand{\E}{{\mathcal{E}}}
\newcommand{\G}{{\mathcal{G}}}
\newcommand{\A}{{\mathcal{A}}}
\newcommand{\C}{{\mathcal{C}}}
\newcommand{\modn}{{\rm \hspace{0.1cm} (mod \hspace{0.1cm} }}
\newcommand{\bmu}{{\textbf \mu}}
\newcommand{\hloc}{{\rm loc}}
\newcommand{\f}{{\mathfrak f}}
\newcommand{\g}{{\mathfrak g}}
\newcommand{\h}{{\mathfrak h}}
\newcommand{\cost}{{\mathfrak c}}

\newcommand\ddfrac[2]{\frac{\displaystyle #1}{\displaystyle #2}}

\vskip 0.5cm

\title{Divisibility questions in commutative algebraic groups}
\author{Laura Paladino\footnote{Partially supported by Istituto Nazionale di Alta Matematica F. Severi with grant ``Assegno di ricerca Ing. Giorgio Schirillo''; partially supported by Max Planck Institute for Mathematics in Bonn}}
\date{  }
\maketitle

\vskip 1.5cm

\begin{abstract}
Let $k$ be a number field, let $\mathcal{A}$ be a commutative algebraic group defined over $k$ and  let 
$p$ be a prime number. Let $\A[p]$ denote the $p$-torsion subgroup of $\A$.  We give some sufficient conditions for the local-global divisibility by $p$ in $\A$
and the triviality of the Tate-Shafarevich group $\Sha (k,\A[p])$.  When $\A$ is a principally polarized abelian variety, those conditions 
imply that the elements of the Tate-Shafarevich group $\Sha(k,\A)$ are divisible by $p$ in the Weil-Ch\^atelet 
group $H^1(k,\A)$ and the local-global principle for divisibility by $p$ holds in $H^r(k,\A)$, for all $r\geq 0$.

\end{abstract}

\section{Introduction} \label{sec0}
We consider two local-global problems, strongly related, that recently arose as   generalizations of some classical questions. 
Let $\A$ be  
 a commutative algebraic group defined over a number field $k$.
 Let $\bar{k}$ be the algebraic closure of $k$ and let $M_k$ be the set of places $v$ of $k$.
 For every positive integer $q$, we denote by $\A[q]$  the
$q$-torsion subgroup of $\A$ and by $k(\A[q])$  the number field obtained by adding to $k$ the coordinates of the $q$-torsion points of $\A$. 
It is well-known that $\A[q]\simeq (\ZZ/q\ZZ)^{n}$, for some positive integer $n$ depending only on $\A$. 
The Galois group $G:=\Gal(k(\A[q])/k)$ is then isomorphic to the image of the representation of the absolute Galois group $G_k:=\Gal(\bar{k}/k)$ in the general linear group $\GL_{n}(\ZZ/q\ZZ)$. The behaviour of the $G$-module $\A[q]$ is related to the answer to the following question, known as \emph{Local-Global Divisibility Problem} in commutative algebraic groups.

\par\bigskip\noindent  \begin{problem} \label{prob1} Let $\A$ be a commutative algebraic group defined over a number field $k$. Let $P\in {\mathcal{A}}(k)$ and let $q$ be a positive integer. Assume that for all but finitely many valuations $v\in k$, there exists $D_v\in {\mathcal{A}}(k_v)$ such that $P=qD_v$. Is it possible to conclude that there exists $D\in {\mathcal{A}}(k)$ such that $P=qD$?
\end{problem}

\par\bigskip\noindent This problem was stated in 2001 by Dvornicich and Zannier and its formulation was motivated by a particular case of the famous Hasse Principle on quadratic forms and by the Grunwald-Wang Theorem (see \cite{DZ}, \cite{DZ2} and \cite{DL}).

It is well-known that the vanishing of $H^1(G,\A[q])$ is a sufficient condition for the local-global divisibility by $q$ (see \cite{DZ}). Anyway, this condition is
not necessary and the obstruction to the local-global principle for divisibility by $q$ in $\A$ is given by a subgroup of  $H^1(G,\A[q])$, denoted by
$H^1_{\loc}(G,\A[q])$ (see Section \ref{sec1} for further details), that contains an isomorphic copy of the Tate-Shafarevic group $\Sha(k,\A[q])$.
In fact in \cite{San} Sansuc introduced a modified Tate-Shafarevich group $\Sha_{\omega}(k,\A[p])$, which contains $\Sha(k,\A[p])$
and it is isomorphic to $H^1_{\loc}(G,\A[q])$ by Lemma 1.2 of the same paper. 

Clearly a solution to Problem \ref{prob1} for all powers $p^l$ of prime numbers $p$ is sufficient to get an answer for all integers $q$, 
by  the unique factorization in $\ZZ$ and B\'{e}zout's identity.

 In the case of elliptic curves the problem has been widely studied since 2001.
The answer is affirmative when $q$ is a prime $p$ (see \cite{DZ} and \cite{Won}), for every $k$. For all powers $2^n$, with $n\geq 2$ there are explicit 
counterexamples over $\QQ$  (see \cite{Cre}, \cite{DZ2}, \cite{Pal}, \cite{GR}) and for  $3^n$, with $n\geq 2$ there are explicit 
counterexamples both over $\QQ$  (see \cite{Cre}) and over $\QQ(\z_3)$ (see \cite{Pal}, \cite{Pal3}). For all powers of a prime $p\geq 5$
the answer is affirmative over $\QQ$ (see \cite{PRV2}). Moreover if $k$ does not contain the field $\QQ(\z_p+\z_p^{-1})$, where $\z_p$ is a p-th root of the unity, 
the local-global divisibility holds for all powers of every prime $p>( 3^{[k:\Q] /2} + 1 )^2$ (see \cite{PRV}).

An aswer to the local-global divisibility by an odd prime number $p$ in the algebraic tori, has been given in \cite{Ill}. The answer is positive in every torus of dimension $n <3(p-1)$.
On the contrary for all $n\geq 3(p-1)$ there are counterexamples. \par
This last result in particular shows that if $n\geq 3$, then even the local-global divisibility by $p$ may fail.
\normalcolor So we are sure that \emph{ the local-global divisibility by $p$ does not hold in general}. 
This is also underlined by Dvornicich and Zannier in \cite[\S 3]{DZ}, for $n\geq 3$. They construct some examples of subgroups of $\GL_n(q)$, with $n\in\{3,4\}$,
and show that the local-global divisibility by $p$ fails in $\A$ over $k$ if $\Gal(k(\A[p])/k)$ has a representation in $\GL_n(q)$, whose image is one of their examples
 (see also Subsection \ref{subsec3} for further details). A priori their example for $n=4$ can appear as a counterexample to the local-global divisibility by $p$ even
in abelian varieties. Anyway, they have no evidence that their examples really realize representations of some Galois group $\Gal(k(\A[p])/k)$ and then the situation is
not clear yet (see also \cite{Ran}).  \par For abelian varieties, some sufficient conditions to have the local-global divisibility by $p^n$, for every
$n\geq 1$ appear in \cite{GR2} and in \cite{GR3}.
Anyway, for a general abelian variety $\A$, one of the conditions is $H^1(\Gal(k(\A[p])/k),\A[p])=0$.  
So the question about the divisibility by $p$ remain in fact open,  since, as stated above, the vanishing of  $H^1(\Gal(k(\A[p])/k),\A[p])$
widely assures the local-global divisibility by $p$. Let $\z_p$ be a $p$-th root of the unity. 
Only in the case of principally polarized abelian varieties defined over number fields $k$ linearly disjoint from $\QQ(\z_p)$,
the condition $H^1(\Gal(k(\A[p])/k),\A[p])=0$ is replaced  by some conditions concerning all the fields $k(P)$ generated by the coordinates of a
point $P$, as $P$ varies in $\A[p]$ (see \cite[Theorem 3]{GR2}).  
\par In the end there are no criteria to establish the validity of  the local-global divisibility by $p$ in a general abelian
variety, as well as in a general commutative algebraic group $\A$.

\par\bigskip\noindent Here we prove that, excluding some particular cases when $p$ is small with respect to $n$, 
 the strongest obstruction to the validity of the Hasse principle for divisibility by $p$ is essentially the
reducibility of $\A[p]$ as $\Gal(\bar{k}/k)$-module. In particular if $\A[p]$ is irreducible as a $N$-module,
for every subnormal subgroup $N$ (see Definition \ref{subn} below) of $\Gal(k(\A[p])/k)$, not contained in its center, 
then we get an affirmative answer for the divisibility by all $p>\ddfrac{n}{2}+1$, for every $n$. We will call such
a module a very strongly irreducible one, in accordance with the well-know definition of
strongly irreducible module, that we are going to recall (see \cite[Definition 1.1]{Bra}).

\begin{D}
We say that an irreducible $\Gamma$-module $M$ is \emph{strongly irreducible}   
if $M$  is an  irreducible $N$-module,
for every normal subgroup $N\leq \Gamma$, not contained in the center $Z(\Gamma)$.
\end{D}

\bigskip \noindent We recall the definition of subnormal subgroup (see \cite[\S 6, pag.150]{Rot} and see also \cite{Rob})
and then we state the definition of very strongly irreducibility, that concerns
subnormal  subgroups of $\Gamma$ and not only normal subgroups of $\Gamma$. 

\begin{D} \label{subn}
A subgroup $N_0$ of a group $\Gamma$ is called a subnormal group if  there exists a normal series

$$N_0 \leq N_1 ... \leq N_j\leq \Gamma,$$

\noindent where $j$ is a positive integer (in other words $N_j\unlhd \Gamma$ and  $N_i\unlhd N_{i+1}$, for every $0\leq i\leq j-1$).
\end{D}

\begin{D}
 We say that an irreducible $\Gamma$-module $M$ is \emph{very strongly irreducible}   
if $M$ is an  irreducible $N$-module,
for every subnormal subgroup $N\leq \Gamma$, not contained in the center $Z(\Gamma)$.
\end{D}

\bigskip
\noindent We prove the following statement.

\begin{thm} \label{Pal17}
Let $p$ be a prime number. Let $k$ be a number field and let $\A$ be a commutative algebraic group defined over $k$, with $\A[p]\simeq (\ZZ/p\ZZ)^n$. 
Assume that $\A[p]$ is a very strongly irreducible $G_k$-module or a direct sum of very strongly irreducible $G_k$-modules.
If $p>\ddfrac{n}{2}+1$, then the local-global divisibility by $p$ holds in $\A$ over $k$ and $\Sha(k,\A[p])=0.$

\end{thm}

\bigskip

There are evidences that the bound  $p>n/2+1$ appearing in the statement of Theorem \ref{Pal17} could be sharp in many cases. 
At the end of the paper we will show a bound which is likely sharp in all cases (see Remark \ref{rem_sharp} and Theorem \ref{sharp}).
We have not presented this bound in the statement of Theorem \ref{Pal17}, with the aim of giving a simpler and more elegant bound for each $n$.

\bigskip
\noindent By the proof of Theorem \ref{Pal17}, we will also deduce the following results.

\begin{cor} \label{Nori2}
Let $p$ be a prime number. Let $k$ be a number field and let $\A$ be a commutative algebraic group defined over $k$, where $\A[p]\simeq (\ZZ/p\ZZ)^n$. 
Let $p\geq n+1$. If  the  absolute Galois group $G_k$ acts on $\A[p]$ as a subgroup  of an  extraspecial group, then  
$H^1(G, \A[p])=0$.

\end{cor}

\begin{cor} \label{Nori}
Let $p$ be a prime number. Let $k$ be a number field and let $\A$ be a commutative algebraic group defined over $k$, where $\A[p]\simeq (\ZZ/p\ZZ)^n$. 
Let $p> \max\left\{63,\left(\ddfrac{n}{2}+1\right)^2\right\}$. If  the  absolute Galois group $G_k$ acts on $\A[p]$ as a subgroup  of a group of Lie type in cross characteristic, then  
$H^1(G, \A[p])=0$.

\end{cor}

\noindent 
The triviality of $H^1(G, \A[p])$ is assured by a deep theorem proved by Nori 
(see \cite[Theorem E]{Nori}) in many cases, i.e. whenever $G_k$ acts semisimply on $\A[p]\simeq (\ZZ/p\ZZ)^n$ and
$p$ is greater than a constant $c(n)$,  depending only on $n$. Anyway the constant is not explicit. 
In our statement,  in the cases when $G_k$ acts on $\A[p]\simeq (\ZZ/p\ZZ)^n$ as a subgroup of $\GL_n(p)$ isomorphic to a subgroup of an almost simple group in cross characteristic or
an extraspecial group, we can respectively give explicit bounds $p>\max\left\{63,\left(\ddfrac{n}{2}+1\right)^2\right\}$ and $p> n+1$ to get the triviality of that first cohomology group.

\bigskip In the case when $\A$ is an abelian variety, with dual $\A^{\vee}$, 
the triviality of  $\Sha(k,\A[p]^{\vee})$ implies $\Sha(k,\A)\subseteq p H^r(k,\A)$, for every positive integer $r$ (see \cite[Theorem 2.1]{Cre2}). When
$\A$ and $\A^{\vee}$ are isomorphic (i.e. when $\A$ is principally polarized), then 
the vanishing of $\Sha(k,\A[p])$ itself implies $\Sha(k,\A)\subseteq p H^r(k,\A)$, for all $r\geq 1$. Such an inclusion is a sufficient and necessary condition 
to have an affirmative answer to the following second and more general local-global problem.

  \begin{problem} \label{prob2} Let $A$ be a commutative algebraic group defined over a number field $k$.
  Let $q$ be a positive integer and let $\sigma\in H^r(k,\A)$. Assume that for all $v\in M_k$ there exists $\tau_v\in H^r(k_v,\A)$
  such that $q\tau_v=\sigma$. Can we conclude that there exists $\tau \in H^r(k,\A)$, such that $q\tau=\sigma$?
  \end{problem}

 \noindent Problem \ref{prob2} was firstly considered by Cassels for $r=1$ in the case when $\A$ is an elliptic curve $\E$ (see \cite[Problem 1.3]{Cas}). In particular Cassels
 questioned if the elements of the Tate-Shafarevich group $\Sha(k,\E)$ 
  were divisible by $p^l$ in the Weil-Ch\^{a}telet group $H^1(k,\E)$, for all $l$. Tate produced soon an affirmative
  answer for divisibility by $p$  (see \cite{Cas2}).

\begin{pro}[Tate, 1962]  \label{Tate}
Problem 2 has an affirmative answer when $r=1$, $\A$ is an elliptic curve $\E$ and $q=p$ is a prime number.
\end{pro} 

\noindent The question for powers $p^l$, with $l\geq 2$ remained open for decades. The mentioned affirmative results to Problem
 \ref{prob1} in elliptic curves imply an affirmative answer to Problem \ref{prob2}, since the proofs show the triviality of
the corresponding Tate-Shafarevich group. So Cassels' question has an affirmative answer for all $p\geq 5$ in elliptic curves over $\QQ$
and for all $p>( 3^{[k:\Q] /2} + 1 )^2$ in elliptic curves over $k$. On the contrary, for powers of $p\in \{2,3\}$ the answer is negative over $\QQ$ by
\cite{Cre2}.
\par The problem was afterwards considered for abelian varieties by Ba\v{s}makov (see \cite{Bas}) and 
lately by \c{C}iperiani and Stix, who gave some sufficient conditions for a positive answer (see \cite{CS}). 
One of their conditions is again the vanishing of $H^1(\Gal(k(\A[p])/k), \A[p])$, so in particular the question for divisibility by $p$ is still open.
 In \cite{Cre} Creutz also proved that for every prime $p$, there exists an  abelian variety
$A$ defined over $\QQ(\z_p)$ such that $\Sha(k,A) \not\subseteq  pH^1(k,A)$. Thus in abelian varieties  of dimension strictly greater than 1, 
even the local-global divisibility by $p$ may fail for Problem 2, as well as for Problem 1. 
As a consequence of Theorem \ref{Pal17}, we have the following statement.

 \begin{cor} \label{P17_gal}
Let $p$ be a prime number.  Let $\A$ be an abelian variety principally polarized of dimension $g$.
Assume that $\A[p]$ is a very strongly irreducible $G_k$-module or a direct sum of very strongly irreducible $G_k$-modules and 
$p> \ddfrac{n}{2}+1$, then  the local-global divisibility by $p$ holds in $H^r(k,\A)$, for all $r\geq 0$.

\end{cor}

\noindent Corollary \ref{P17_gal} can be considered a generalization of Tate's Proposition \ref{Tate} 
to all commutative algebraic groups. 

\par\bigskip About the structure of this paper, a few preliminary known results in the theory of groups and in local-global divisibility  are stated in the next section.
Then we proceed with the proof of Theorem \ref{Pal17}. We firstly show the validity of the local-global divisibility by $p$ and the
triviality  of $\Sha(k,\A[p])$ in some particular cases, i.e.,  when the image of the representation
of $G_k$ in $\GL_n(p)$  is
an extension of one of its normal subgroups with trivial local cohomology (see Lemma \ref{extension}) 
 or when it is a whole classical group (see Lemma \ref{whole}) or when $\A[p]$ has a tensor product decomposition (see Lemma \ref{tensor} and Lemma \ref{tensor3}). 
 Then we show that Theorem \ref{Pal17} holds for $n\in\{2,3\}$. 
At the end we give a proof of Theorem \ref{Pal17} for a general $n$ and we deduce Corollary \ref{Nori2} and Corollary \ref{Nori}.

\section{Preliminary results} \label{sec1}

We recall some known results about local-global divisibility and about group theory, that will be useful in the following.

\bigskip\noindent We keep the notation introduced in Section 1. Thus $k$ denotes a number field and $\A$ denotes  a commutative algebraic group, defined over $k$.
From now on let $q:=p^l$, where $p$ is a prime number and $l$ is a positive integer. As introduced before, the $q$-torsion subgroup of $\A$ will be denoted by  $\A[q]$ 
and the number field generated over $k$ by the coordinates of the points in $\A[q]$ will be denoted by $F:=k(\A[q])$. 
The $q$-torsion subgroup $\A[q]$ of $\A$ is a $G_k$-module, where $G_k$ denotes the absolute Galois group $\Gal(\bar{k}/k)$.
We have $\A[q]\simeq (\Z/q\Z)^{n}$, for a certain $n$ depending only on $\A$.  Thus $G_k$ acts over $\A[q]$ as a subgroup of 
$\GL_{n}(\Z/q\Z)$ isomorphic to $G=\textrm{Gal}(k(\A[q])/k)$.
We still denote by $G$ the representation of $G_k$ in  $\GL_{n}(\Z/q\Z)$.
If $q=p$ is a prime number, in particular $G\leq \GL_{{n}}(p)$. When $\A$ is an abelian variety
of dimension $g$, we have $n=2g$.

\bigskip In \cite{DZ} Dvornicich and Zannier proved that the answer to the local-global question for divisibility by $q$ of points in $\A(k)$
is linked to the behaviour of a subgroup of $H^1(G,\A[q])$ named \emph{first local cohomology group} and defined as follows (see \cite[\S 2]{DZ}).

\begin{D} \label{loc_cond}
 A cocycle $\{Z_{\sigma}\}_{\sigma\in G}$ representing a class in $H^1(G,\G[q])$ satisfies the
\emph{local conditions} if, for every $\sigma\in G$, there exists $A_{\sigma}\in \G[q]$ such that
$Z_{\sigma}=(\sigma-1)A_{\sigma}$. The subgroup of $H^1(G,\G[q])$ formed by all the classes of the cocycles satisfying the local conditions
is the \emph{first local cohomology group} $H^1_{\textrm{loc}}(G,\G[q])$. Equivalently

$$H^1_{\textrm{loc}}(G,\A[q]):=\bigcap_{C\in \Omega} \ker  (H^1(G,\A[q])\xrightarrow{\makebox[1cm]{{\small $res_v$}}} H^1(C,\A[q])),$$

\noindent where $\Omega$ is the set of all cyclic subgroups $C$ of $G$ and $res_v$, as usual, denotes the restriction map. 
\end{D}

\bigskip\noindent Let $\Sigma$ be the subset of $M_k$ containing all the  places $v$ of $K$, that are  unramified in $F$. 
Thus $\Sigma$ contains all but finitely many places $v$ of $k$.
For every $v\in \Sigma$, we denote by $G_v$ the Galois group $\Gal(F_w/k_v)$, where $w$ is a place of $F$ extending $v$. 
By the  \v{C}ebotarev's Density Theorem $\Gal(F_w/k_v)$ varies among all $C\in \Omega$ as $v$ varies in $\Sigma$. 
So we have

\begin{equation} \label{h1loc}
H^1_{\textrm{loc}}(G,\A[q])=\bigcap_{v\in \Sigma} \ker  (H^1(G,\A[q])\xrightarrow{\makebox[1cm]{{\small $res_v$}}} H^1(G_v,\A[q])).
\end{equation}

\noindent By \cite[Lemma 1.2]{San} this group is isomorphic to
a group containing the Tate-Shafarevich group

$$\Sha(k,\A[q]):=\bigcap_{v\in M_k} \ker  (H^1(k,\A[q])\xrightarrow{\makebox[1cm]{{\small $res_v$}}} H^1(k_v,\A[q])).$$

\bigskip\noindent  In particular, the vanishing of $H^1_{\textrm{loc}}(G,\A[q])$ assures the
triviality of $\Sha(k,\A[q])$, that is a sufficient condition to get an affirmative answer to Problem \ref{prob2}, for $r=0$, and in many cases 
for all $r\geq 0$ (see \cite[Theorem 2.1]{Cre2} and Section \ref{sec0}). 
Furthermore the triviality of $H^1_{\textrm{loc}}(G,\A[q])$ is a sufficient condition for an affirmative answer to Problem \ref{prob1}
by \cite[Proposition 2.1]{DZ}.

\bigskip

\noindent \begin{rem} \label{rem_cyc} Observe that if $G$ is cyclic, then $H^1_{\loc}(G, \A[q])=0$.
\end{rem}

\vskip 0.2cm

\noindent The vanishing of  $H^1_{\textrm{loc}}(G,\A[q])$ is strongly related to the
 behaviour of  $H^1_{\textrm{loc}}(G_p,\A[q])$, where $G_p$ is the $p$-Sylow subgroup of $G$ (see \cite{DZ}).

\vskip 0.2cm

\begin{lem}[Dvornicich, Zannier] \label{Sylow}
Let $G_p$ be a $p$-Sylow subgroup of $G$. An element of $H^1_{\hloc}(G, \A[q])$ is zero if and only if
its restriction to $H^1_{\hloc}(G_p, \A[q])$ is zero.
\end{lem}

\bigskip
In our proofs of Theorem \ref{Pal17}, a crucial tool is the use of  Aschbacher's Theorem on the classification of maximal subgroups of $\GL_n(q)$ (see \cite{Asc}).
Aschbacher proved that the maximal subgroups of $\GL_n(q)$  could be divided into 9 specific classes, that we denote by $\C_i$, $1\leq i\leq 9$
(by $\C_9$ we denote a class that was originally called $\mathcal{S}$ by Aschbacher).
For a big $n$, it is a very hard open problem to find  the maximal subgroups of $\GL_n(q)$ of type $\C_9$. We have an explicit
list of such groups only for $n\leq 12$ (see \cite{BHR}). On the contrary, the maximal subgroups of $\GL_n(q)$ of geometric type  (i.e. of class $\C_i$, with $1\leq i\leq 8$) have been described for every $n$ (see \cite{KL}).
We recall some notations in group theory and then we resume the description  of the maximal subgroups of $\GL_n(q)$ of geometric type
 in the following Table 1 (see \cite[Table 1.2.A, $\S$ 3.5 and $\S$ 4.6]{KL}).

\bigskip


\begin{Not} 
Let $n,l$ be  positive integers, let $p$ be a prime number and let $q=p^l$. We denote by $\FF_q$ the finite field with $q$ elements.  
Let $\omega_q$ be a primitive element of $\F_q^*$. We use the standard notations for the special linear group
 $\SL_n(q)$, the projective special linear group $\PSL_n(q)$,  
the  unitary group $\Uni_n(q)$, the projective unitary group $\PUni_n(q)$,
the symplectic group $\Sp_n(q)$, the projective symplectic group $\PSp_n(q)$, the simmetric group $S_n$ and the alternating group $A_n$.
By $C_n$ we denote a cyclic group of order $n$  
and by $p^{1+2n}$ an extraspecial group of order $p^{1+2n}$.
Furthermore, if $n$ is odd or both $n$ and $q$ are even, then we denote by $\textrm{O}_n(q)$ the orthogonal group and by $\textrm{SO}_n(q)$ the  special orthogonal group. 
If $n$ is even and $q$ is odd  we denote by  (see \cite{BHR})
\begin{description}
\item[ ]  $\GO_n^+(q)$ the stabilizer of the non-degenerate symmetric bilinear antidiagonal form (1,...,1);
\item[ ]  $\SO_n^+(q)$ the subgroup of $\GO_n^+(q)$ formed by the matrices with determinant 1;
\item[ ]  $\GO_n^-(q)$ the stabilizer of non-degenerate symmetric bilinear form $I_n$, when $n\equiv 2 (\textrm{mod 4})$ and
$q\equiv 3 (\textrm{mod 4})$ and the stabilizer of non-degenerate symmetric bilinear diagonal form $(\omega_q,1,...,1)$, when $n\not\equiv 2 (\textrm{mod 4})$ and
$q\not\equiv 3 (\textrm{mod 4})$;
\item[ ]  $\SO_n^-(q)$ the subgroup of $\GO_n^-(q)$ formed by the matrices with determinant 1.
\end{description}

\noindent  For $n$ even and  $\epsilon\in\{+,-\}$, we denote by $\Omega_n^{\epsilon}(q)$ the subgroup of index 2 of $O_n^{\epsilon}$, obtained as the kernel of the spinor norm.
 \end{Not}

\bigskip
\begin{Not}
 Let $A, B$ be two groups. We denote by

\begin{description}
\item[ ] $A\rtimes B$, the semidirect product of  $A$ with $B$  (where $A\trianglelefteq A\rtimes B$);
\item[ ] $A\circ B$, the central product of $A$ and $B$;
\item[ ] $A\wr B$, the wreath product of $A$ and $B$;
  \item[ ] $A.B$, a group $\Gamma$ that is an extension of  its normal subgroup $A$ with the group $B$ (then $B\simeq \Gamma/A$),
in the case when we do not know if it is
a split extension or not;
\item[ ] $A^{.}B$, a group $\Gamma$ that is non-split extension of its normal subgroup $A$ with the group $B$ (then $B\simeq \Gamma/A$);
\item[ ] $A:B$, a group $\Gamma$ that is a split extension of its normal subgroup $A$ with the group $B$
(then $B\simeq \Gamma/A$ and $\Gamma\simeq A\rtimes B$).
\end{description}
 \end{Not}

\begin{center}

\begin{tabular}{|c|c|c|}

\hline
 type & description & structure  \\
\hline
$\C_1$ &   \begin{minipage}[t]{7cm} stabilizers of totally singular or non$–$singular subspaces \end{minipage}
&  maximal parabolic group\\
\hline
 $\C_2$ &  \begin{minipage}[t]{7cm} stabilizers of direct sum decompositions \par $V =\bigoplus_{i=1}^{r} V_i$, with each $V_i$
of dimension $t$ \end{minipage} &  \begin{minipage}[t]{6cm}  \centerline{$\GL_t(q)\wr S_r, n=rt$} \end{minipage} \\
\hline
$\C_3$ & \begin{minipage}[t]{7cm} stabilizers of extension fields of $\FF_q$ of prime index $r$   \end{minipage} &  $\GL_t(q^r).C_r, n=rt, r$ prime \\
\hline
 $\C_4$ & \begin{minipage}[t]{7cm} stabilizers of tensor product decompositions $V=V_1\otimes V_2$  \end{minipage} &  $\GL_t(q)\circ \GL_r(q), n=rt$ \\
\hline
 $\C_5$ & \begin{minipage}[t]{7cm} stabilizers of subfields of $\F_q$ of prime index $r$  \end{minipage} & $\GL_n(q_0)$, $q=q_0^r$, $r$ prime \\
\hline
 \begin{minipage}[t]{0.5cm} \vskip 0.2cm \centerline{$\C_6$} \end{minipage} & \begin{minipage}[t]{7cm} 
normalizers of symplectic-type  $r$-groups \par ($r$ prime) in absolutely  irreducible \par representations \end{minipage} & \begin{minipage}[t]{6.8cm} 

\vskip 0.2cm
 \centerline{ $(C_{q-1}\circ r^{1+2t}).\Sp_{2t}(r)$, $n=r^{t}, r$  prime, $r\neq p$ }  

\end{minipage}   \\
\hline
 $\C_7$ & \begin{minipage}[t]{7cm} stabilizers of tensor product decompositions \par $V=\bigotimes_{i=1}^t V_i, \textrm{dim}(V_i)=r$  \end{minipage} &
$ \underbrace{ (\GL_r(q)\circ ... \circ \GL_r(q)) }_{t} .S_t, n=r^t $\\
 \hline
 \begin{minipage}[t]{0.5cm} \vskip 0.2cm \centerline{$\C_8$} \end{minipage} & \begin{minipage}[t]{7cm}  \vskip 0.2cm \centerline{classical subgroups} \end{minipage} & \begin{minipage}[t]{6cm}    \centerline{$\Sp_n(q)$, $n$ even } \par
  \centerline{$\textrm{O}_n^{\epsilon}(q)$, $q$ odd}  \par
 \centerline{$\Uni_n(q^{\frac{1}{2}}),$ $q$ a square} \end{minipage} \\
\hline
\multicolumn{3}{c}{  }\\

\multicolumn{3}{c}{Table 1: Maximal subgroups of $\GL_n(q)$ of geometric types}
\end{tabular}

 \end{center}

\bigskip\noindent Although we generally do not know explicitly the maximal subgroups of type $\C_9$, by
 Aschbacher's Theorem, we have such a characterization of them:

 \bigskip\noindent ``if $\Gamma$ is a maximal subgroup of $\GL_n(q)$ of class $\C_9$ and $Z(\Gamma)$ denotes its center, then for
 some nonabelian simple group $T$, the group
$\Gamma/ Z(\Gamma)$ is almost simple with socle $T$; in this
case the normal subgroup $ Z(\Gamma) T$ acts absolutely
irreducibly, preserves no nondegenerate classical
form, is not a subfield group, and does not contain
$\SL_n(q)$.''

\bigskip\noindent We will use this description in our proof of Theorem \ref{Pal17}. 
Furthermore, for very small integers $n$ there are a few subsequent and more explicit versions of  Aschbacher's Theorem,
that describe exactly the maximal subgroups of class $\C_9$. To prove  Theorem \ref{Pal17} we will use the classification of the maximal subgroups of $\SL_n(q)$ appearing in \cite{BHR}, for $n\leq 12$.

\bigskip\noindent From now on we will say that a subgroup $G$ of $\GL_n(q)$ (respectively of $\SL_n(q)$) is of class $\C_i$ or of type $\C_i$, with $1\leq i\leq 9$, if $G$ is contained in
a maximal subgroup of $\GL_n(q)$ (respectively of $\SL_n(q)$) of class $\C_i$. 

\section{Proof of Theorem \ref{Pal17}}

\noindent  The proof of Theorem \ref{Pal17}
follows by the proof of the next slightly more general statement, with the only difference in the hypotheses that we assume $G\leq \GL_n(p^m)$,
 for some positive integer $m$ (instead of simply $\GL_n(p)$).
This more general assumption considering powers of $p$ in lieu of $p$ will be especially useful when $G$ is of type $\C_3$
and it is isomorphic to a subgroup of $\GL_t(p^r).C_r$, with $n=tr$, for some prime number $r$.

\begin{thm} \label{P1_bis}
Let $p$ be a prime number. Let $k$ be a number field and let $\A$ be a commutative algebraic group defined over $k$. 
Assume that $G=\Gal(k(\A[p])/k)$ is isomorphic to a subgroup of $\GL_{n}(p^m)$, for some positive integers $n, m$.
If $\A[p]$ is a very strongly irreducible $G$-module or a direct sum of very strongly irreducible $G$-modules and 
$p> \ddfrac{n}{2}+1$, then the local-global divisibility by $p$ holds in $\A$ over $k$ and $\Sha(k,\A[p])=0$.
\end{thm}

\bigskip

When $n=2$, $m=1$ and $p\neq 2$, the conclusion of Theorem \ref{P1_bis} follows immediately by
Chevalley's Theorem  on the classification of the commutative algebraic groups
in characteristc 0 (see for example \cite{Ser2}), combined with the mentioned results in \cite{DZ} and in \cite{Ill}. Anyway, when $m>1$, or when $m=1$, $p= 2$ and
$\A$ an algebraic torus, there are no similar results in the literature, even for $n=2$.  
 Thus, we will give a proof for the more general case when $G\leq \GL_2(p^m)$, with $m\geq 1$,  for $n=2$ too. That will be a part of the base
of the induction for the general case. In fact, for $n=2$ we can prove a stronger result than  Theorem \ref{P1_bis}, since
it suffices to assume that $\A[p]$ is an irreducible $G$-module (or a direct sum of irreducible $G$-modules), 
as we will see in Subsection \ref{subsec2} (have a look at Proposition \ref{lemma_caso2}). 

\bigskip  We firstly prove a very useful lemma. In fact it covers many cases
when $G$ is isomorphic to a subgroup of $\GL_n(p^m)$, which acts irreducibly on $\A[p]$ and it is an extension of a nontrivial normal
subgroup $S$ of its with trivial local cohomology. Observe that here the hypothesis of irreducibility (and not very strongly irreducibility) is sufficient.
Of course every statement proved for an irreducible $G$-module, holds for a very strongly irreducible $G$-module too (as well as for a strongly irreducible $G$-module).

\begin{lem} \label{extension}  
Let $p$ be a prime number and $n,m$ positive integers. Let $\A$ be a commutative algebraic group defined over a number field $k$. 
Assume that $G=\Gal(k(\A[p])/k)$ acts irreducibly on $\A[p]$ as a subgroup of $\GL_{n}(p^m)$, which is an extension $S.J$, where
$S$ is a nontrivial normal subgroup of $G$. 
If  $H^1_{\loc}(S,\A[p])=0$, then $H^1_{\loc}(G,\A[p])=0$.
\end{lem}

\begin{proof} 
 Consider the inflation-restriction
exact sequence

\begin{equation} \label{infres} 0\rightarrow H^1(J,\A[p]^S)\rightarrow H^1(G,\A[p])\rightarrow H^1(S,\A[p])^{J}. \end{equation}

\noindent Let $\{Z_{\sigma}\}_{\sigma\in G}$ be a cocycle representing a class $Z$ in $H^1_{\loc}(G,\A[p])$. Since the
restriction map sends $H^1_{\loc}(G,\A[p])$ to $H^1_{\loc}(S,\A[p])$
and $H^1_{\loc}(S,\A[p])=0$, then  $Z$ is in the kernel of the restriction, that coincides with the image of the inflation by the
exactness of the sequence. By assumption $\A[p]$ is an irreducible $G$-module, then it has no nontrivial proper $G$-submodules.
We are going to show that $\A[p]^S$ is a $G$-submodule of $\A[p]$.
Let $W\in \A[p]^S$ and $\sigma\in G$. We have $\sigma(W)\in \A[p]^S$ if and only if $\rho \sigma(W) =\sigma(W)$,
for all $\rho\in S$, if and only if $\sigma^{-1}\rho\sigma(W)=W$, for all $\rho\in S$. Since $S$ is a normal
subgroup of $G$, then   $\sigma^{-1}\rho\sigma\in S$, for all $\rho\in S$ and $\sigma\in G$.  
Thus $\sigma^{-1}\rho\sigma (W)=W$, for all $\rho\in S$ and $\A[p]^S$ is a $G$-submodule of $\A[p]$.
Owing to $\A[p]$ being irreducible, we have that $\A[p]^S$ is trivial or that $\A[p]^S=G$. If $\A[p]^S$ is trivial,
then $H^1(J,\A[p]^S)=0$, implying $Z=0$ and  $H^1_{\loc}(G,\A[p])=0$.
If $\A[p]^S=G$, then $S$ is trivial and we have a contraddiction with our assumption.
Then $H^1_{\loc}(G,\A[p]^S)=0$. 
\end{proof}

\noindent \begin{rem} \label{extension_rem} Observe that the conclusion of Lemma \ref{extension}
hold as well if $S$ is trivial, when $H^1_{\loc}(J,\A[p])=0$. In fact, if $S$ is trivial, then $G=J$. 
In the proof of Theorem \ref{P1_bis} we will often apply Lemma \ref{extension} when $G$ is a subgroup of  $S.C_t$,
for some positive integer $t$. In this case if $S$ is trivial, then $G$ is a subgroup of $J=C_t$. Since
 $H^1_{\loc}(C_t,\A[p])=0$, then  $H^1_{\loc}(G,\A[p])=0$ in any case.
\end{rem}

\noindent The hypothesis that $\A[p]$ is an irreducible $G$-module in the statement of Lemma \ref{extension} can be replaced by the
hypothesis that $\A[p]$ is an irreducible $S$-module. In fact if $\A[p]$ is an irreducible $S$-module, then
in particular it is an irreducible $G$-module (if $\A[p]$ has a nontrivial $G$-submodule $V$, then $V$ is also a nontrivial $S$-submodule of its). 

\begin{cor} \label{extension_sn}  
Let $p$ be a prime number and let $n,m$ be positive integers. Let $\A$ be a commutative algebraic group defined over a number field $k$. 
Assume that $G=\Gal(k(\A[p])/k)$ is isomorphic to a subgroup of $\GL_{n}(p^m)$ that is an extension $S.J$, where $S$ is a nontrivial normal subgroup of $G$, acting irreducibly on
$\A[p]$. If  $H^1_{\loc}(S,\A[p])=0$, then  $H^1_{\loc}(G,\A[p])=0$.
\end{cor}

\bigskip\noindent With the next corollary we observe that the hypothesis about the irreducibility of $\A[p]$ as an $S$-module
in Corollary \ref{extension_sn} can be replaced by the hyphotesis of the existence of a certain particular element in $S$.   

\begin{cor} \label{extension_sn_cor}  
Let $p$ be a prime number and let $n,m$ be positive integers. Let $\A$ be a commutative algebraic group defined over a number field $k$. 
Assume that $G=\Gal(k(\A[p])/k)$ is isomorphic to a subgroup of $\GL_{n}(p^m)$ that is an extension $S.J$, with $S$ nontrivial. 
If  $H^1_{\loc}(S,\A[p])=0$ and there exist $\rho\in S$ such that $\rho-1$
is invertible, then  $H^1_{\loc}(G,\A[p])=0$.
\end{cor}

\begin{proof} 
 Obviously $\A[p]^S$ is an $S$-module.  
By using \eqref{infres} and the  argument in Lemma \ref{extension},   it suffices to prove
that $\A[p]^S$ is a trivial $S$-module and then $H^1_{\loc}(J,\A[p]^S)=0$. 
 This follows from the fact that $\rho-1$ is invertible
and it fixes no element in $\A[p]^S$.
\end{proof}

The next Lemma will be very useful in the case when $\A[p]\simeq (\FF_p)^n$ has a tensor product decomposition  
$V_1 \bigotimes V_2$, where $V_1$, $V_2$ are vector spaces
 over $\F_{p}$, with dimension respectively $t$ and $r$, where $rt=n$. 
Observe that this case does not occur when $n$ is a prime.
\bigskip

\begin{lem} \label{tensor}
Let $\A[p]=V_1 \bigotimes V_2$, where $V_1$, $V_2$ are vector spaces
 over $\F_{p^m}$, with dimension respectively $t$ and $r$.  Assume that $G=G_t\circ G_r$ acts on $\A[p]$
as a subgroup of $\GL_t(p^m)\circ \GL_r(p^m)$ and that $\A[p]$ is a strongly irreducible $G$-module. 
Assume that $H^1_{\loc}(G_t,V_1)=0$ and $H^1_{\loc}(G_r,V_2)=0$. 
Then $H^1_{\loc}(G,\A[p])=0$.
\end{lem}

\begin{proof} 
We firstly recall some basic facts.
Let $x\in V_1$ and $y\in V_2$.  Let
$\GL(V_i)$ be the group of the linear transformations of $V_i$, for $i\in \{1,2\}$ and $\GL(V)$
the group of the linear transformations of $V=V_1\otimes V_2$. As in \cite[\S 4.4]{KL},
for every $\sigma\in \GL(V_1)$ and $\tau\in \GL(V_2)$ we denote by $\sigma\otimes \tau$
the map acting on elements $x\otimes y\in V$ as $(\sigma\otimes \tau)(x\otimes y)=
\sigma(x)\otimes \tau(y)$ and then acting on the whole $V$, with the action extended by linearity. 
The group  $\GL(V_1)\otimes \GL(V_2)$ is formed by the maps $\sigma\otimes\tau$, with $\sigma\in \GL(V_1)$ and $\tau\in \GL(V_2)$.
 We recall that $\GL_t(p^m)\circ \GL_r(p^m)\simeq \GL(V_1)\otimes \GL(V_2)\leq \GL(V)$ by \cite[\S 4.4]{KL}.
A central product $\Gamma$ of two groups is a quotient of their direct
product by a subgroup of its center. Then every subgroup of $\Gamma$ is a central product of two groups too (where one of the two
groups or both can be trivial). 
In particular, if $G=G_t\circ G_r$, then $G_t$ is  a subgroup of $\GL_t(p^m)$ acting on $V_1$ and $G_r$ is  a subgroup of $\GL_t(p^m)$ 
acting on $V_2$  (see also \cite[\S 4.4]{KL}). We denote by $\overline{G_t}$ and $\overline{G_r}$ the images of $G_t$ and
respectively $G_r$ in the quotient $G=G_t\circ G_r$ of the direct product $G_t\times G_r$. Observe that in particular 
 $\overline{G_t}$ is a quotient of $G_t$ by a subgroup $Z_1$ of its center  and $\overline{G_r}$ is a quotient of $G_r$ by
a subgroup $Z_2$ of its center, such that $Z_1\simeq Z_2$.
Being $\overline{G_t}$ a normal subgroup of $G$, 
we can use the inflation-restriction exact sequence


\begin{equation} \label{infres3} 0\rightarrow H^1(\overline{G_t},\A[p]^{\overline{G_r}})\rightarrow H^1(G,\A[p])\rightarrow H^1(\overline{G_r},\A[p])^{\overline{G_t}}
\rightarrow H^2(\overline{G_t},\A[p]^{\overline{G_r}}). \end{equation}

\noindent 
Since $\overline{G_r}$ is a normal subgroup of $G$, with the same argument used in the proof of Lemma \ref{extension},
we have that $\A[p]^{\overline{G_r}}$ is a $G$-submodule of $\A[p]$.
Being $\A[p]$ a strongly irreducible $G$-module, then $\A[p]^{\overline{G_r}}$ is the whole $\A[p]$ or it is trivial.
If $\A[p]^{\overline{G_r}}=\A[p]$, then $\overline{G_r}$ is trivial. We have $G=G_t\otimes \langle I_r\rangle \simeq \overline{G_t}$.
Observe that in particular $G$  acts on $\A[p]$ as $\overline{G_t}$ and
$H^1_{\loc}(\overline{G_t},\A[p])\simeq H^1_{\loc}(G,\A[p])$.
Let $Z_{\sigma\otimes I_{r}}$ be a cocycle representing a class in $H^1_{\loc}(G,\A[p])$, such that

\begin{equation} \label{Z_tr} Z_{\sigma\otimes I_{r}}=\left(
    \begin{array}{c}
         x_{\sigma,1}  \\
        \vdots \\
  x_{\sigma,tr} \\
    \end{array}
  \right). \end{equation}

\medskip\noindent Let $1\leq j\leq r$. We define $r$ cocycles $Z_{\sigma,j}$ representing
$r$ classes in $H^1_{\loc}(G_t,V_1)$, by setting

\begin{equation} \label{Z_trj}  Z_{\sigma,j}=\left(
    \begin{array}{c}
         x_{\sigma,j}  \\
 x_{\sigma,r+j}  \\
 x_{\sigma,2r+j}  \\
        \vdots \\
 x_{\sigma,(t-1)r+j}  \\
    \end{array}
  \right). \end{equation}

\medskip\noindent For the reader's convenience, we firstly show that the cocycles $Z_{\sigma,j}$, $1\leq r\leq t$ are well defined and satisfy the local conditions when
$t=r=2$. 
Assume that $Z_{\sigma\otimes I_2}$ is a cocycle representing a class in $H^1_{\loc}(G,\A[p])$, where $\sigma\in G_t$. We have

$$Z_{\sigma\otimes I_2}=  \left(
    \begin{array}{c}
       x_{\sigma,1} \\
        x_{\sigma,2} \\
x_{\sigma,3} \\
x_{\sigma,4}\\
    \end{array}
  \right),$$

\medskip\noindent for some $ x_{\sigma,1},  x_{\sigma,2},  x_{\sigma,3},  x_{\sigma,4}\in \FF_p$.
Let 

$$\sigma = \left(
    \begin{array}{cc}
       \beta_{1,1} &  \beta_{1,2} \\
        \beta_{2,1} &  \beta_{2,2} \\
    \end{array}
  \right),$$

\medskip\noindent with $\beta_{i,j} \in \ZZ/p\ZZ$, for every $1\leq i,j \leq 2$. Then

$$\sigma\otimes I_2= \left(
    \begin{array}{cccc}
       \beta_{1,1} &  0 & \beta_{1,2} & 0 \\
0 & \beta_{1,1} &  0 & \beta_{1,2}  \\
        \beta_{2,1} & 0 & \beta_{2,2} & 0\\
 0 & \beta_{2,1} & 0 & \beta_{2,2}\\
    \end{array}
  \right)$$

\medskip\noindent and 

$$\sigma\otimes I_2 - I_2\otimes I_2= \left(
    \begin{array}{cccc}
       \beta_{1,1} -1 &  0 & \beta_{1,2} & 0 \\
0 & \beta_{1,1} -1 &  0 & \beta_{1,2}  \\
        \beta_{2,1} & 0 & \beta_{2,2} -1 & 0\\
 0 & \beta_{2,1} & 0 & \beta_{2,2} -1\\
    \end{array}
  \right).$$

\medskip\noindent  Since   $Z_{\sigma\otimes I_2}$ satisfies the
local conditions, then there exist $y_{\sigma,1}, y_{\sigma,2}, y_{\sigma,3}, y_{\sigma,4}\in \FF_p$, depending on $\sigma$, such that

$$ \left(
    \begin{array}{cccc}
       \beta_{1,1} -1 &  0 & \beta_{1,2} & 0 \\
0 & \beta_{1,1} -1 &  0 & \beta_{1,2}  \\
        \beta_{2,1} & 0 & \beta_{2,2} -1 & 0\\
 0 & \beta_{2,1} & 0 & \beta_{2,2} -1\\
    \end{array}
  \right)  \left(
    \begin{array}{c}
       y_{\sigma,1} \\
        y_{\sigma,2} \\
y_{\sigma,3} \\
y_{\sigma,4}\\
    \end{array}
  \right) =  \left(
    \begin{array}{c}
       x_{\sigma,1} \\
        x_{\sigma,2} \\
x_{\sigma,3} \\
x_{\sigma,4}\\
    \end{array}
  \right).$$

\medskip\noindent Thus

\begin{equation} \label{tensor_loc_cond} \left(
    \begin{array}{c}
       (\beta_{1,1} -1)y_{\sigma,1} + \beta_{1,2}y_{\sigma,3} \\
 (\beta_{1,1} -1)y_{\sigma,2} + \beta_{1,2}y_{\sigma,4}  \\
         \beta_{2,1} y_{\sigma,1} + (\beta_{2,2}-1)y_{\sigma,3}\\
  \beta_{2,1}y_{\sigma,2} + (\beta_{2,2}-1)y_{\sigma,4}\\
    \end{array}
  \right)   = \left(
    \begin{array}{c}
       x_{\sigma,1} \\
        x_{\sigma,2} \\
x_{\sigma,3} \\
x_{\sigma,4}\\
    \end{array}
  \right). \end{equation}

\medskip\noindent We define 

$$Z_{\sigma,1}:=\left(
    \begin{array}{c}
       x_{\sigma,1} \\
x_{\sigma,3} \\
    \end{array}
  \right) \hspace{1cm} 
Z_{\sigma,2}:=\left(
    \begin{array}{c}
       x_{\sigma,2} \\
x_{\sigma,4} \\
    \end{array}
  \right).$$

\medskip\noindent We are going to show that $Z_{\sigma,1}, Z_{\sigma,2}$ are cocycles
of $G_t$ with values in $V_1$, representing two classes in $H^1_{\loc}(G_t,V_1)$.
Let $\tau\in G_t$, such that

$$\tau = \left(
    \begin{array}{cc}
       \gamma_{1,1} &  \gamma_{1,2} \\
        \gamma_{2,1} &  \gamma_{2,2} \\
    \end{array}
  \right),$$

\medskip\noindent with $\gamma_{i,j}\in \FF_p$, for all $1\leq i,j\leq 2$. Then

$$\sigma\tau = \left(
    \begin{array}{cc}
       \beta_{1,1} \gamma_{1,1} +  \beta_{1,2} \gamma_{2,1} &   \beta_{1,1} \gamma_{1,2} +  \beta_{1,2} \gamma_{2,2} \\
       \beta_{2,1} \gamma_{1,1} +  \beta_{2,2} \gamma_{2,1} &   \beta_{2,1} \gamma_{1,2} +  \beta_{2,2} \gamma_{2,2} \\
    \end{array}
  \right)$$

and

$$\sigma\tau\otimes I_2= \left(
    \begin{array}{cccc}
        \beta_{1,1} \gamma_{1,1}+  \beta_{1,2} \gamma_{2,1}  &  0 &  \beta_{1,1} \gamma_{1,2} +  \beta_{1,2} \gamma_{2,2} & 0 \\
0 &  \beta_{1,1} \gamma_{1,1}+  \beta_{1,2} \gamma_{2,1}  &  0 &  \beta_{1,1} \gamma_{1,2} +  \beta_{1,2} \gamma_{2,2} \\
         \beta_{2,1} \gamma_{1,1} +  \beta_{2,2} \gamma_{2,1}  & 0 &  \beta_{2,1} \gamma_{1,2} +  \beta_{2,2} \gamma_{2,2} & 0\\
 0 &   \beta_{2,1} \gamma_{1,1} +  \beta_{2,2} \gamma_{2,1}  & 0 &  \beta_{2,1} \gamma_{1,2} +  \beta_{2,2} \gamma_{2,2} \\
    \end{array}
  \right)$$

\medskip\noindent Let

$$Z_{\sigma\tau\otimes I_2}= \left(
    \begin{array}{c}
       x_{\sigma\tau,1} \\
        x_{\sigma\tau,2} \\
x_{\sigma\tau,3} \\
x_{\sigma\tau,4}\\
    \end{array}
  \right).$$

\medskip\noindent We have that $Z_{\sigma\tau\otimes I_2}=Z_{(\sigma\otimes I_2)(\tau\otimes I_2)}=
Z_{(\sigma\otimes I_2)}+(\sigma\otimes I_2)Z_{(\tau\otimes I_2)},$ i.e.

\begin{equation} \label{tensor_cocycle} \left(
    \begin{array}{c}
       x_{\sigma\tau,1} \\
        x_{\sigma\tau,2} \\
x_{\sigma\tau,3} \\
x_{\sigma\tau,4}\\
    \end{array}
  \right)=\left(
    \begin{array}{c}
       x_{\sigma,1} \\
        x_{\sigma,2} \\
x_{\sigma,3} \\
x_{\sigma,4}\\
    \end{array}
  \right)+ \left(\begin{array}{c}
       \beta_{1,1}x_{\tau,1} + \beta_{1,2}x_{\tau,3} \\
 \beta_{1,1}x_{\tau,2} + \beta_{1,2}x_{\tau,4}  \\
         \beta_{2,1} x_{\tau,1} + \beta_{2,2}x_{\tau,3}\\
  \beta_{2,1}x_{\tau,2} + \beta_{2,2}x_{\tau,4}\\
    \end{array}
  \right). \end{equation}

\medskip\noindent By \ref{tensor_cocycle}, in particular we deduce $Z_{\sigma\tau,1}=
Z_{\sigma,1}+\sigma(Z_{\tau,1})$ (by considering the 1st and the 3rd rows in \eqref{tensor_cocycle})
and $Z_{\sigma\tau,2}=
Z_{\sigma,2}+(\sigma)Z_{\tau,2}$ (by considering the 2nd and the 4th rows in \eqref{tensor_cocycle}). Moreover by \eqref{tensor_loc_cond} we have  

$$Z_{\sigma,1}=(\sigma-1)\left(
    \begin{array}{c}
       y_{\sigma,1} \\
y_{\sigma,3} \\
    \end{array}
  \right) \hspace{1cm} \textrm{and}\hspace{1cm} 
Z_{\sigma,2}=(\sigma-1)\left(
    \begin{array}{c}
       y_{\sigma,2} \\
y_{\sigma,4} \\
    \end{array}
  \right).$$

\medskip\noindent Thus $Z_{\sigma,1}$ and $ Z_{\sigma,2}$ satisfy the local conditions. Being $H^1_{\loc}(G_t,V_1)=0$, then
there exist $y_1,y_2,y_3,y_4\in \FF_p$ that do not depend on $\sigma$ such that

$$Z_{\sigma,1}=(\sigma-1)\left(
    \begin{array}{c}
       y_{1} \\
y_{3} \\
    \end{array}
  \right) \hspace{1cm} \textrm{and}\hspace{1cm} 
Z_{\sigma,2}=(\sigma-1)\left(
    \begin{array}{c}
       y_{2} \\
y_{4} \\
    \end{array}
  \right).$$

\medskip\noindent We deduce

$$ \left(
    \begin{array}{c}
       (\beta_{1,1} -1)y_{1} + \beta_{1,2}y_{3} \\
 (\beta_{1,1} -1)y_{2} + \beta_{1,2}y_{4}  \\
         \beta_{2,1} y_{1} + (\beta_{2,2}-1)y_{3}\\
  \beta_{2,1}y_{2} + (\beta_{2,2}-1)y_{4}\\
    \end{array}
  \right)   = \left(
    \begin{array}{c}
       x_{\sigma,1} \\
     x_{\sigma,2} \\
x_{\sigma,3} \\
x_{\sigma,4}\\
    \end{array}
  \right), $$

\medskip\noindent implying  
$H^1_{\loc}(G,\A[p])=0$.
For general $t$ and $r$ we may apply the
same argument.   Let $Z_{\sigma\otimes I_r}$ as in \eqref{Z_tr} and let 
 $Z_{\sigma,j}$ as in \eqref{Z_trj}, for every $1\leq j\leq r$. If $\sigma\in G_t$, such that

$$\sigma = \left(
    \begin{array}{ccc}
       \beta_{1,1} & \hdots & \beta_{1,r} \\
 \vdots &  & \vdots \\
        \beta_{t,1} &  \hdots &\beta_{t,r} \\
    \end{array}
  \right),$$

\medskip\noindent with $\beta_{i,j}\in \FF_p$, for all $1\leq i\leq t$, $1\leq j\leq r$, then

$$\sigma\otimes I_r= \left(
    \begin{array}{c|c|c}
       \beta_{1,1} I_r &  \hdots & \beta_{1,2}I_r  \\
\hline
 \vdots & \ddots & \vdots \\
\hline
        \beta_{t,1} I_r &  \hdots &\beta_{t,r} I_r \\
    \end{array}
  \right).$$





\medskip\noindent With the same calculations showed when $t=r=2$, one easily verifies that
$Z_{\sigma,j}\in H^1_{\loc}(G_t,V_1)$, for every $1\leq j\leq r$. Thus $Z_{\sigma,j}$
is a coboundary, for every $1\leq j\leq r$. By proceeding as in the case when $t=r=2$, this implies 
that  $Z_{\sigma\otimes I_{tr}}$ is a coboundary too. We deduce  $H^1_{\loc}(G,\A[p])=0$.\par
If $\A[p]^{\overline{G_r}}=0$, then   $H^1(\overline{G_t},\A[p]^{\overline{G_r}})$ and
 $H^2(\overline{G_t},\A[p]^{\overline{G_r}})$ are both trivial. From the inflation-restriction exact sequence
\eqref{infres3}, we have  $H^1(G,\A[p])\simeq H^1(\overline{G_r},\A[p])^{\overline{G_t}}$.
In particular $H^1(G,\A[p])$ is isomorphic to a subgroup of $H^1(\overline{G_r},\A[p])$. By identifying an element
$\rho\in\overline{G_r}$ with $\Id_t\otimes \rho$, we can proceed as above to get $H^1_{\loc}(\overline{G_r},\A[p])=0$, by using
$H^1_{\loc}(G_r,V_2)=0$.  This implies $H^1_{\loc}(G,\A[p])=0$.
\end{proof}

\bigskip \noindent With a very similar proof, we have the following result about the first local cohomology group
of a group $G$ acting on a tensor product in the way described by Steinberg's Tensor Product Theorem (see \cite{Ste}),
with the only difference that here the representation of $G$ is not projective.

\begin{lem} \label{tensor3}
Let $t$, $r$ be positive integers. 
Let $V=\bigotimes_{i=1}^t V_i$, where, for every $1\leq i\leq t$, $V_i$ is a vector space  over $\F_{p^m}$ of the same dimension $r$.
Let $\f: x\mapsto x^p$ be the Frobenius map of $\F_{p^m}$. For every $\sigma_i\in \GL( V_i)$ and every integer $0\leq j\leq m$, 
let $\f^j(\sigma_i)$ denote
the matrix obtained by raising the entries of $\sigma_i$ to their $p^j$-th powers. In particular $\f^0(\sigma)=\sigma$. Let $\rho\in S_t$.
Assume that $G$ is a subgroup of $\GL(V)$ formed by elements $\f^{j_1}(\sigma_1)\otimes  ... \otimes \f^{j_t}(\sigma_t)$  
acting on $v_1 \otimes v_2\otimes ... \otimes v_t\in V$ in the following way $(\f^{j_1}(\sigma_1)\otimes ... \otimes \f^{j_t}(\sigma_t))(v_1 \otimes v_2\otimes ... \otimes v_t)=
(\f^{j_1}(\sigma_1))(v_{\rho(1)}) \otimes ...\otimes (\f^{j_t}(\sigma_t))(v_{\rho(t)})$, where $0\leq j_i\leq m$, for every $1\leq i\leq t$. Consider the action of $G$ extended by linearity to the whole $V$. For a fixed $i$, let $G_i$ be the subgroup of $\GL(V_i)$ formed by the maps $\sigma_i$ (appearing in the $i$-th position) as above. 
Assume that $V$ is a strongly irreducible $G$-module.
If  $H^1_{\loc}(G_i, V_i)=0$, for all $1\leq i\leq t$, then
$H^1_{\loc}(G,V)=0$.
\end{lem}

\begin{proof} 
When $t=2$, we can use the same arguments as in the proof of Lemma \ref{tensor}. Here the only differences are
that the entries $\beta_{i,j}$ in $\sigma$ are now raised to their $p^j$-th powers, for some $0\leq j\leq m$ and the
$x_{\sigma,i}$, with $1\leq i\leq n$, are permuted. Then we can use induction to get the conclusion for $t>2$.
\end{proof}

\noindent Observe that Lemma \ref{tensor3} holds when $\rho=\Id$ too. Moreover
Lemma  \ref{tensor3} holds when $j_i=0$, for all $1\leq i\leq t$ (this happens in particular when $m=1$,
as, for instance, when $V=\A[p]=(\FF_p)^n$).
For completeness, we state explicitly the particular case when $\A[p]=(\FF_p)^n$,
and $\rho=\Id$, 
since we will use it in the proof of Theorem \ref{P1_bis}.

\begin{cor} \label{tensor_2}
Let $\A[p]=\bigotimes_{i=1}^t V_i$, where $V_i$ is a vector space  over $\F_{p}$ of the same dimension $r$, for every $1\leq i\leq t$.
Assume that $G$ is a subgroup of $\GL(V)$ formed by elements $\sigma_1\otimes  ... \otimes \sigma_t$  
acting on $v_1 \otimes v_2\otimes ... \otimes v_t\in \A[p]$ in the following way $(\sigma_1\otimes ... \otimes\sigma_t)(v_1 \otimes v_2\otimes ... \otimes v_t)=
\sigma_1(v_1) \otimes ...\otimes \sigma_t(v_t)$. Consider the action of $G$ extended by linearity to the whole $\A[p]$. For a fixed $i$, let $G_i$ be the subgroup of $\GL(V_i)$ formed by the maps $\sigma_i$ (appearing in the $i$-th position) as above. Assume that $\A[p]$ is a strongly irreducible $G$-module.
If  $H^1_{\loc}(G_i, V_i)=0$, for all $1\leq i\leq t$, then
$H^1_{\loc}(G,\A[p])=0$.
\end{cor}

\begin{rem} \label{tensor4}
Note that Lemma \ref{tensor}, Lemma \ref{tensor3} and Corollary \ref{tensor_2} hold as well when $G\leq \PGL(V)$, instead of $G\leq \GL(V)$. 
Thus the conclusion of Lemma \ref{tensor3} and Corollary \ref{tensor_2} hold as well in the case of a group
$G$ acting on $\A[p]$ as prescribed by Steinberg's Tensor Product Theorem  (see \cite{Ste}).
\end{rem}

\bigskip
 The next step is to  study the triviality of $H^1_{\loc}(G,\A[p])$, when 
 $G$ is one of the whole classical groups $\SL_t(p^m)$,
$\Sp_{2t}(p^m)$, $\O_t(p^m)$, $\O_{2t}^{\epsilon}(p^m)$, with $\epsilon \in \{+,-\}$, or $\Uni_t(p^m)$, for some positive integer $t$. 

\bigskip

\begin{lem} \label{whole}
 Let $p>3$. 

\begin{description}
\item[1)] Assume that $G$ is one of the groups $\SL_t(p^m)$, $\O_{2t}^{\epsilon}(p^m)$, with $\epsilon \in \{+,-\}$, 
$\O_t(p^m)$, or $\Uni_t(p^m)$, for some positive integer $t$. Assume that $\A[p]$ is a strongly irreducible $G$-module.
Then  $H^1_{\loc}(G,\A[p])=0$, for all $p>2$.
\item[2)] Assume that $G=\Sp_{2t}(p^m)$, for some positive integer $t$. Assume that $\A[p]$ is a strongly irreducible $G$-module.
Then  $H^1_{\loc}(G,\A[p])=0$, for all $p> t +1$.
\end{description}
\end{lem}

\begin{proof}
 Let $S$ be the quotient by scalar matrices of the classical matrix group $G$. In particular $S$ is
a classical Chevalley group. We have $S=G/Z(G)$, where $Z(G)$ is the center of $G$. Since no scalar matrix
has order dividing $p$, for every $p>2$, then the $p$-Sylow subgroup of $Z(G)$ is trivial
and $H^1_{\loc}(Z(G),\A[p])=0$, by Lemma \ref{Sylow}. If $Z(G)$ is nontrivial,
then we are in the situation described in Corollary \ref{extension_sn_cor} and
 $H^1_{\loc}(G,\A[p])=0$. If $Z(G)$ is trivial, then $G=S$. 
We will prove $H^1_{\loc}(S,\A[p])=0$. It is well-known that the group $S$ is a simple group (see \cite{Wil})
and it is also called a group of Lie type. 
 Assume first that $S$ is neither a projective symplectic group $\PSp_{2t}(p^m)$, 
nor a projective unitary group $\PUni_t(p^m)$. By \cite{CPS} (see in particular Table (4.5) at page 186) and \cite{CPS2} 
(see in particular Table (4.3) at page 193), the cohomology
group $H^1(S,\A[p])$ is trivial, for all $p>3$, when $t$ is the possible minimal degree of the representation of $S$
and when $t$ is the dimension of the Lie algebra with automorphism group $S$ (in this case $S$ has a natural
representation in dimension $t$ and  this often coincides with the representation of $S$ with minimal degree).
If the representation of $S$ is neither the minimal nor the natural one, then we can proceed as follows.  Let
$L(\lambda)$ denote the irreducible $S$-module of highest weight $\lambda$.
In \cite[Theorem 1.2.3]{Georgia} the authors prove that for all $p>3$ the first cohomology group
$H^1(S,L(\lambda))$ is trivial, when $\lambda$ is a fundamental
dominant weight (or it is less than a fundamental dominant weight). 
In particular, we have  $H^1_{\loc}(S,L(\lambda))=0$.
In 1950 Chevalley proved that whenever $M$ is an irreducible $S$-module, then $M=L(\lambda)$, for
some dominant weight $\lambda$ (see \cite{His} and \cite{Che}). In particular, since we are assuming that $\A[p]$
is irreducible, then $\A[p]=L(\lambda)$, for
some dominant weight $\lambda$. In addition, every dominant weight is a positive integer linear combination of
fundamental dominant weights, and it is well-known in the theory of Lie groups that 
this implies a decomposition $L(\lambda)=\otimes_{i=1}^s L(\omega_i)$, where
$s$ is a positive integer ($s\leq n$) and $\omega_i$ is a fundamental dominant weight, for every $1\leq i\leq s$.
Thus $H^1_{\loc}(S,\A[p])\simeq H^1_{\loc}(S, \otimes_{i=1}^s L(\omega_i))$, for certain fundamental
dominant weights $\omega_i$. If $S$ preserves the tensor product decomposition $\otimes_{i=1}^s L(\omega_i)$, as in the
case of a group of class $\C_4$  (see \cite[\S 4.4]{KL}), then $S$ acts on $\A[p]$ as described
in Lemma \ref{tensor}. More 
generally the group $S$ can act on $\A[p]$ as prescribed by Steingberg's Tensor Product Theorem (see \cite{Ste}),
i.e. in the way described in Lemma  \ref{tensor3} (see also Remark \ref{tensor4}). We have
$S=  \otimes_{i=1}^s S_i$, with $S_i$ acting on $L(\omega_i)$, for every $1\leq i\leq s$.
In all those cases,
since Theorem 1.2.3 in \cite{Georgia} assures the triviality of $H^1_{\loc}(S_i,L(\omega_i))$, for all $1\leq i\leq s$, 
then the triviality of $H^1_{\loc}(S,\A[p])$ follows by the mentioned lemmas. 
Now assume that $S=\PUni_t(p^m)$, for some positive integer $t$. Then the $p$-Sylow subgroups of $S$ 
coincides with the $p$-Sylow subgroups of $\PSL_t(p^m)$. Since we have proved
 $H^1_{\loc}(\PSL_t(p^m),\A[p])=0$, for all $p>3$, then  $H^1_{\loc}(\PUni_t(p^m),\A[p])=0$, for all $p>3$ again, by Lemma
\ref{Sylow}.
If $S$ is a projective symplectic group $\PSp_n(p^l)$, then,  for $p>3$ the cohomology
 $H^1(S,\A[p])$ is trivial again, when $2t$ is the possible minimal degree of the representation of $S$
and when $2t$ is the dimension of the Lie algebra with automorphism group $S$, by \cite{CPS} (see again Table (4.5) at page 186) and \cite{CPS2} 
(see  Table (4.3) at page 193). On the contrary
if the representation of $S$ is neither the minimal nor the natural one, then it is not true in general that
the cohomology $H^1(S,L(\lambda))$ vanishes. 
As above, let $L(\lambda)$ denote the irreducible $S$-module of highest weight $\lambda$ 
and assume $L(\lambda)=\otimes_{i=1}^s L(\omega_i)$, where
$s$ is a positive integer and $\omega_i$ is a fundamental weight, for every $1\leq i\leq s$. Again
$S=  \otimes_{i=1}^s S_i$, with $S_i$ acting on $L(\omega_i)$, for every $1\leq i\leq s$.
Observe that since $L(\omega_i)$
is a vector subspace of $L(\lambda)$ over $\FF_p$, in particular $s\leq 2t$ and  $i\leq 2t$.
Let $t+1=b_0+b_1p+b_2p^2+ ... +b_jp^j$, where $j$ is a positive integer 
and $b_h\in \FF_p$, for $0\leq h\leq j$. Then the mentioned Theorem 1.2.3 in
\cite{Georgia} implies $H^1(S_i,L(\omega_i))=0$, for every $i\neq 2b_hp^h$.
On the other hand, when $i= 2b_hp^h$, then $H^1(S_i,L(\omega_i))\neq 0$.
Observe that if $p> t+1$, then $t+1=b_0$ and $2b_0=2(t+1)=2t+2>2t$. 
Thus $H^1(S_i,L(\omega_i))=0$, for every $i$, and we may
repeat the arguments used above for projective special linear groups and projective orthogonal groups to get $H^1_{\loc}(S,\A[p])=0$.
\end{proof}

The next remark, will allow us to deal with subgroups of $\SL_n(p^m)$, instead of
$\GL_n(p^m)$.

\begin{rem} \label{SL_n}
Let $G$ be a subgroup of $\GL_{n}(p^m)$ and
let $\widetilde{G}:=G\cap \SL_n(p^m)$. Since $|\GL_n(p^m)|=(p^m-1)|\SL_n(p^m)|$,
then the $p$-Sylow subgroup of $\GL_n(p^m)$ coincides with the $p$-Sylow subgroup of $\SL_n(p^m)$.
By Lemma \ref{Sylow}, we have  $H^1_{\textrm{loc}}(G,\A[p^m])=0$ if and only if
$H^1_{\loc}(\widetilde{G},\A[p^m])=0$. Moreover $\SL_n(p^m)$ is a normal subgroup of $\GL_n(p^m)$ and then,
if $\A[p]$ is a very strongly irreducible $G$-module, then $\A[p]$ is a very strongly irreducible $\widetilde{G}$-module too. 
Therefore from now on we assume  $G\leq \SL_n(p^m)$, without loss of generality.
\par Observe that when $\A[p]$ is an irreducible $G$-module, the vanishing of $H^1_{\loc}(\widetilde{G},\A[p^m])$
implies the vanishing of $H^1_{\textrm{loc}}(G,\A[p^m])$ by Lemma \ref{extension} (and Remark \ref{extension_rem}). In fact $\GL_n(p^m)$
is an extension of $\SL_n(p^m)$ by a cyclic group. 
\end{rem}

\bigskip\noindent 
From now on we will often assume, without loss of generality, that $G$ is a proper subgroup
of $\SL_n(p^m)$.

\bigskip\noindent 
For $n\in \{ 2, 3\}$ we give a proof  of Theorem \ref{P1_bis} based on a case by case analysis of the possible maximal
subgroups of $\SL_n(p^m)$. Then we proceed with the proof of Theorem \ref{Pal17} for a general $n$.

\subsection{The case when $n=2$} \label{subsec2}

In this section we consider algebraic groups $\A$ such that $G$ is isomorphic to a subgroup of $\GL_2(p^m)$, for some positive integer $m$. In particular
this is the case when $\A[p]\cong (\ZZ/p\ZZ)^2$. As stated above if $m=1$ and $p\neq 2$, then the conclusion of Theorem \ref{P1_bis} follows immediately by
Chevalley's Theorem  on the classification of the commutative algebraic groups
in characteristc 0 (see for example \cite{Ser2}), combined with the mentioned results in \cite{DZ} and \cite{Ill}. Anyway, when $m> 1$, or $m=1$, $p= 2$ and
$\A$ an algebraic torus, there are no similar results in the literature.  
 Thus, here we give a proof for the more general case when
$G\leq \GL_2(p^m)$, with $m\geq 1$. 
We use the classification of the maximal subgroups of $\SL_2(q)$ 
appearing in \cite{BHR}, for $q=p^m$, that we partially recall in the next lemma. 

\begin{lem} \label{n=2lem} Let $q=p^m$, where $p$ is a prime number and $m$ is a positive integer. 
The maximal subgroups of $\SL_2(q)$ of type $\C_i$, with $2\leq i\leq 9$  are

\begin{description}
 \item[(a)] a subgroup of type $\C_2$, the generalized quaternion group $Q_{2(q-1)}$ of order $2(q-1)$, with $q$ odd, $q\neq 5$;
  \item[(b)] a subgroup of type $\C_2$, the dihedral group $D_{2(q-1)}$ of order $2(q-1)$, with $q$ even;
  \item[(c)] a subgroup of type $\C_3$, the generalized quaternion group $Q_{2(q+1)}$,  of order $2(q+1)$, for $q$ odd;
  \item[(d)] a subgroup of type $\C_3$, the dihedral group $D_{2(q+1)}$ of order $2(q+1)$, for $q$ even;
  \item[(e)] a subgroup of type $\C_5$, the group $\SL_2(q_0).C_2$, with $q=q_0^2$;
  \item[(f)] subgroup of type $\C_5$,  the group $\SL_2(q_0)$, with $q=q_0^r$, for $q$ odd, $r$ an odd prime;
  \item[(g)] subgroup of type $\C_5$,  the group $\SL_2(q_0)$, with $q=q_0^r$, for $q$ even, $q_0\neq 2$, $r$ prime;
  \item[(h)] a group of type $\C_6$, the group $2^{1+2}.S_3$, for $q=p\equiv \pm 1 \modn 8)$;
 \item[(i)] a group of type $\C_6$, the group $2^{1+2}:C_3$, for $q=p\equiv \pm 3, 5, \pm 11, \pm 13, \pm 19 \modn 40)$;
  \item[(j)] a group of type $\C_9$, the group $C_2^{\cdot}A_5$, $q=p\equiv \pm 1 \modn 10)$ or $q=p^2$, with $p\equiv \pm 3 \modn 10)$.
\end{description}

\end{lem}

Indeed, for $n=2$ we are going to prove the following stronger result than Theorem \ref{P1_bis}, with
the assumption that $\A[p]$ is an irreducible $G$-module or  a direct sum of irreducible $G$-modules.

\begin{pro} \label{lemma_caso2}
Let $p$ be a prime number. Let $k$ be a number field and let $\A$ be a commutative algebraic group defined over $k$. 
Assume that $G=\Gal(k(\A[p])/k)$ is isomorphic to a subgroup of $\GL_{2}(p^m)$, for some positive integer $m$.
If $\A[p]$ is an  irreducible $G_k$-module or a direct sum of irreducible $G_k$-modules,  
then the local-global divisibility by $p$ holds in $\A$ over $k$ and $H^1_{\loc}(G,\A[p])=0$.

\end{pro}

\begin{proof}
As already noticed in \cite{PRV},
for every group $\Gamma$ and every direct sum of two $\Gamma$-modules $M_1$ and $M_2$,
one has  $H^1_{\textrm{loc}}(\Gamma, M_1\times M_2)\simeq   H^1_{\textrm{loc}}(\Gamma, M_1) \bigoplus H^1_{\textrm{loc}}(\Gamma, M_2) $.
Thus $H^1_{\textrm{loc}}(G,.)$ is an additive functor and it suffices to prove the
statement when $\A[p]$ is an irreducible $G_k$-module, to get an answer even in the case when $\A[p]$ is a direct sum of 
 irreducible $G_k$-modules. Thus, we may assume without loss of generality that $G$ is not of class $\C_1$ and it is
 isomorphic to one of the subgroups of $\SL_2(p^m)$ listed in Lemma \ref{n=2lem}. In cases
\textbf{(a)} (resp. \textbf{(c)}), $G$ is a subgroup of the generalized quaternion group   $Q_{2(q-1)}$ (resp. $Q_{2(q+1)}$).
The group $Q_{2(q-1)}$ (resp. $Q_{2(q+1)})$ is an extension $C_2.C_{q-1}$ (resp. $C_2.C_{q+1}$) of a cyclic group of order 2, with a cyclic group of order
$q-1$ (resp. $q+1$). Then $G$ is cyclic or it is an extension of two cyclic groups. Since the local cohomology of a cyclic group is trivial, then
by Lemma \ref{extension} and Remark \ref{extension_rem}, we have $H^1_{\textrm{loc}}(G,\A[p])=0$ (recall that we are assuming that $\A[p]$ is an irreducible $G$-module.
In cases   \textbf{(b)},   \textbf{(d)}, \textbf{(h)}, \textbf{(i)} and \textbf{(j)}, for every $p\geq 2$, the $p$-Sylow subgroup of 
$G$ is either trivial or cyclic (recall that cases \textbf{(h)}, \textbf{(i)} and \textbf{(j)} may occur only if $p\neq 2$). By Lemma \ref{Sylow}, we have $H^1_{\textrm{loc}}(G,\A[p])=0$.
Suppose that we are in case  \textbf{(e)}. For every $p> 2$, the $p$-Sylow subgroup $G_p$ of $G$ is
a subgroup of  $\SL_2(q_0)$, where $q=q_0^2$. Thus, without loss of generality, we may assume that $G$ is a
subgroup of $\SL_2(q_0)$. If $G=\SL_2(q_0)$, then $H^1_{\textrm{loc}}(G,\A[p])=0$, by Lemma \ref{whole}. 
Assume that $G$ is a proper subgroup of $\SL_2(q_0)$. If
$G$ is still of type $C_5$, then $G$ is isomorphic to a subgroup of $\SL_2(q_1)$, where
$q_0=q_1^2$. Again, if $G=\SL_2(q_1)$, then by Lemma \ref{whole}, we have
$H^1_{\textrm{loc}}(G,\A[p])=0$. We may assume that $G$ is a proper subgroup
of $\SL_2(q_1)$ and so on. Since $m$ is finite, at a certain point we will find
that either $G$ is of type $\C_i$, with $i\neq 5$, or $G$ is trivial. If $G$ is
of type $\C_i$, with $i\neq 5$, because of our assumption that $\A[p]$ is a very strongly irreducible $G$-module, 
then $G$ is isomorphic to one of the subgroups listed in cases \textbf{(a)},  \textbf{(b)},  \textbf{(c)},  \textbf{(d)}, \textbf{(h)}, \textbf{(i)} and \textbf{(j)}.
Thus   $H^1_{\textrm{loc}}(G,\A[p])=0$, as above. If $G$ is trivial, then
$H^1_{\textrm{loc}}(G,\A[p])=0$ too. 
The same arguments, combined with Lemma \ref{extension} and Remark \ref{extension_rem}, 
give $H^1_{\textrm{loc}}(G,\A[p])=0$, for $p=2$ too.
Cases  \textbf{(f)} and  \textbf{(g)} are
similar to case  \textbf{(e)}, being $G_p$ a subgroup of $\SL_2(q_0)$, with $q=q_0^r$, with $r$ a prime.
\end{proof}

\vskip 0.2cm
\noindent In particular we have proved Theorem \ref{P1_bis} for $n=2$. 

\subsection{The case when $n=3$} \label{subsec3}

In this section we consider algebraic groups $\A$ such that  $\A[p]\cong (\ZZ/p\ZZ)^3$. As mentioned in Section \ref{sec0}, 
in \cite{DZ} Dvornicich and Zannier underline that
the answer in this case is not obvious. In fact, they show an example in which $H^1_{\loc}(\Gamma,\ZZ/p\ZZ)\neq 0$,
where $\Gamma$ is a subgroup of the $p$-Sylow subgroup of $\GL_3(p)$ of the form

$$ \hspace{1cm}  \left(
    \begin{array}{ccc}
       1 & a & b \\
       0 & 1 & \lambda a \\
       0 & 0 & 1 \\
    \end{array}
  \right), \hspace{0.2cm} a,b\in \ZZ/p\ZZ, \lambda \in \ZZ/p\ZZ^*.$$

\bigskip\noindent Anyway, they have no evidence that this group is really the $p$-Sylow subgroup of a Galois group $\textrm{Gal}(k(\A[p])/k,\A[p])$.
Even in the case when $\Gamma$ would be the $p$-Sylow subgroup of a certain Galois group $\textrm{Gal}(k(\A[p])/k,\A[p])$,
we could get no information about the algebraic group $\A$ for which the local-global
divisibility fails. Here we prove Theorem \ref{P1_bis} for $n=3$.

We use the classification of the maximal subgroups of $\SL_3(q)$ and of  $\SU_3(q)$ appearing in \cite{BHR}.

\begin{lem} \label{M+B+L-cor} Let $q=p^m$, where $p$ is a prime number and $m$ is a positive integer. 
Let $d:=\gcd(q-1,3)$.
The maximal subgroups of $\SL_3(q)$ of type $\C_i$, with $3\leq i\leq 9$  are

\begin{description}

  \item[(a)] a group of type $\C_2$, the group $C_{q-1}^2:S_3$, for $q\geq 5$;
  \item[(b)] a group of type $\C_3$, the group $C_h:C_3$, where $h=q^2+q+1$;
  \item[(c)] a group of type $\C_5$, the group $\SL_3(q_0).C_s$, where $s:=\gcd\left(\frac{q-1}{p-1},3\right)$ 
and $q=q_0^r$, $r$ prime.

  \item[(d)] a group of type $\C_6$, the group $3_{+}^{1+2}:Q_8.C_s$, where $s=\frac{\gcd(q-1,9)}{3}$, $q=p\equiv 1 \modn 3)$ and
the extraspecial group $3_{+}^{1+2}$ is the $p$-Sylow subgroup of $\GL_3(p)$;

 \item[(e)] a group of type $\C_8$, the group $ \SO_3(q)\times C_d$, with $q$ odd;
 \item[(f)] a group of type $\C_8$, the group $\SU_3(q_0)\times C_t$, where $t:=\gcd(p-1,3)$ and $q=q_0^2$;
\item[(g)] a group of type $\C_9$, the group $\PSL_2(7)\times C_d$, for $q=p\equiv 1,2,4 \modn 7), q\neq 2$;
\item[(h)] a group of type $\C_9$, the group $C_3^{\hspace{0.1cm} .}A_6$,  of order $9\cdot 5!$, for $q=p\equiv 1,4 \modn 15)$ or $q=p^2$, with $p\equiv 2,3 \modn 5), p\neq 3$.
\end{description}

\end{lem}

\bigskip

\begin{lem} \label{SU_3} Let $q=p^m$, where $p$ is a prime number and $m$ is a positive integer. 
Let $d:=\gcd(q-1,3)$.
The maximal subgroups of $\SU_3(q)$ of type $\C_i$, with $3\leq i\leq 9$  are

\begin{description}

\item[(e.1)] a group of type $\C_2$, the group $C_{q-1}^2:S_3$, for $q\geq 5$;
  \item[(e.2)] a group of type $\C_3$, the group $C_h:C_3$, where $h=q^2+q+1$, $q\neq 3$;
  \item[(e.3)] a group of type $\C_5$, the group $\SU_3(q_0).C_s$, where $s:=\gcd\left(\frac{q+1}{q+1},3\right)$ and $q=q_0^r$, $r$ prime;
\item[(e.4)] a group of type $\C_8$, the group $\SO_3(q)\times C_d$, $q$ odd and $q\geq 7$;

  \item[(e.5)] a group of type $\C_6$, the group $3_{+}^{1+2}:Q_8.C_s$, where $s=\frac{\gcd(q+1,9)}{3}$, 
the extraspecial group $3_{+}^{1+2}$ is the $p$-Sylow subgroup of $\GL_3(p)$, $q=5$ or $q=p\equiv 2 \modn 3)$ and $q\geq 11$;

\item[(e.6)] a group of type $\C_9$, the group $\PSL_2(7)\times C_d$, $q=p\equiv 3,5,6 \modn 7)$;
\item[(e.7)] a group of type $\C_9$, the group $C_3^{\hspace{0.1cm} .}A_6$, for $q=p\equiv 11, 14 \modn 15)$;
\item[(e.8)] a group of type $\C_9$, the group $C_3^{\hspace{0.1cm} .}A_6^{\cdot}C_2$, (where here $C_2$ is a known specific quotient of $A_6$), for $q=p=5$;
\item[(e.9)] a group of type $\C_9$, the group $C_3^{\hspace{0.1cm} .}A_7$, of order $9\cdot 7\cdot 5!$, for $q=p=5$.
\end{description}

\end{lem}

\bigskip\noindent \textbf{Proof of Theorem \ref{P1_bis} for $n=3$}.
As in the case when $n=2$, since $H^1_{\textrm{loc}}(G,.)$ is an additive functor, we assume without loss of generality that $G$ is not of class $\C_1$ and
then it is a subgroup of the groups listed in
Lemma \ref{M+B+L-cor}.  We are going to show that $H^1_{\loc}(G,\A[p])=0$, for all $p\geq 3$.
In cases  \textbf{(a)}, \textbf{(b)} and  \textbf{(d)} the $p$-Sylow subgroup of $G$ is either trivial or cyclic, for all $p$
(recall that case \textbf{(a)} occur only for $p\geq 5$ and case \textbf{(a)} occur only for $q=p\equiv 1 \modn 3))$.  Therefore $H^1_{\textrm{loc}}(G, \A[p])=0$.
Assume that we are in case \textbf{(g)}. Since this case may occur only for $p\neq 3$, then the $p$-Sylow subgroup of $G$
is isomorphic to a subgroup of $\PSL_2(7)$.  Therefore, for every $p\geq 2$, the $p$-Sylow
subgroup of $G$ is either trivial or cyclic and
 $H^1_{\textrm{loc}}(G, \A[p])=0$ (recall that this case does not hold when $p=2$ too).  
Assume that we are in case  \textbf{(h)}. If $p\geq 3$, then the $p$-Sylow subgroup of $G$ is trivial or cyclic again
(observe that this case does not happen when $p=3$).
 Assume that we are in case  \textbf{(c)}.  For every $p\neq 3$, the $p$-Sylow subgroup of $G$
is a subgroup of $\SL_3(q_0)$. Since $\SL_3(q_0)$ is a normal subgroup of $G$ and $\A[p]$ is a very strongly irreducible $G$-module, we may assume, without loss of generality,
that $G\subseteq \SL_3(q_0)$.  If $G=\SL_3(q_0)$, then
  $H^1(G, \A[p])=0$, by Lemma \ref{whole}. If $G$ is trivial, then  $H^1(G, \A[p])=0$ too. 
So, suppose that $G$ is a non-trivial proper subgroup of $\SL_3(q_0)$.
If $G$ is still of type $\C_5$ in $\SL_3(q_0)$, then $G$ is contained in  $\SL_3(q_1).C_{s_1}$, where
$q_0=q_1^2$ and $s_1=\frac{\gcd(q_0-1,9)}{3}$. Again, we may assume  that $G$ is strictly contained in $\SL_3(q_1)$ and so on. 
Since $q$ is finite, after a finite number of steps we find that $G$ is of type $\C_i$, with $i\neq 5$. We have $H^1_{\textrm{loc}}(G, \A[p])=0$, by the arguments used for
the subgroups of classes $\C_i$, with $i\neq 5$.
Thus  $H^1_{\textrm{loc}}(G, \A[p])=0$ too. If $p=3$, we can use the same argument as for $p\neq 3$, combined with Lemma \ref{extension} and Remark \ref{extension_rem} (by the hypothesis
of the very strongly irreducibility of $\A[p]$).
Assume that we are in case  \textbf{(e)}. 
Again, for all $p$, the $p$-Sylow subgroup $G_p$ of $G$ is contained
in $\SO_3(q)$. Since $\A[p]$ is a very strongly irreducible $G$-module, by Lemma \ref{extension} and Remark \ref{extension_rem} we may assume without loss of generality that
$G$ is contained in $\SO_3(q)$.
 If $G=\SO_3(q)$, then $H^1_{\loc}(G,\A[p])=0$, by Lemma \ref{whole}. 
Assume that $G$ is strictly contained in $\SO_3(q)$.  The group $\SO_3(q)$ is isomorphic to $\SL_2(q)$ (see \cite[Proposition 1.10.1]{BHR}). 
In the proof of Theorem  \ref{P1_bis} for $n=2$, we have
seen that $G$ is a group extension of a cyclic group with another cyclic group or it is a whole special linear group 
or its $p$-Sylow subgroup $G_p$ is either trivial or cyclic. 
In all cases the first cohomology group $H^1_{\loc}(G,\A[p])$ vanishes.
Suppose that we are in case  \textbf{(f)}. By Lemma \ref{extension} and the assumption that $\A[p]$ is very strongly
irreducible, we may 
assume without loss of generality that $G\leq \SU_3(q)$. If $G=\SU_3(q)$, then we get the triviality of 
 $H^1_{\loc}(G,\A[p])$, by Lemma \ref{whole}.  Assume that $G$ is strictly contained  in $\SU_3(q)$.
 Thus we use Lemma \ref{SU_3}. There are only four cases in which the
 subgroups of $\SU_3$ are different from the ones listed in Lemma  \ref{M+B+L-cor}, i.e. cases    \textbf{(e.3)}, \textbf{(e.5)},  \textbf{(e.8)} and  \textbf{(e.9)}.
For all $p$,  in cases \textbf{(e.5)},   \textbf{(e.8)} and  \textbf{(e.9)}, the $p$-Sylow subgroup is either trivial or cyclic. 
Therefore $H^1_{\loc}(G,\A[p])=0$.
We can treat case  \textbf{(e.3)} in the same way as case  \textbf{(c)}.  We have proved that for every possible $G$, if $\A[p]$ is a very strongly
irreducible $G$-module and $p\geq 3$, then the local-global divisibility by $p$ holds in $\A$ over $k$. 
\eop

\bigskip

\subsection{General Case} \label{gen_case}

\bigskip

To prove Theorem \ref{P1_bis} for every $n$, we use the description of the subgroups of $\GL_n(q)$ of geometric type shown in Table 1.  For some classes of groups we also use induction, having already proved the statement for $ n\leq 3$. Moreover we use the description of
subgroups of   $\GL_n(q)$ of class $\C_9$ given in Aschbacher's Theorem and recalled in Section \ref{sec1}. When the $p$-Sylow subgroup of $G$ is isomorphic to $C_p^2$, there are known counterexample
to the local-global divisibility (see the mentioned \cite{DZ}, \cite{DZ2}, \cite{Pal2}) and  even when  the $p$-Sylow subgroup of $G$ is isomorphic to $C_p^3$
there are counterexamples (see \cite{Pal}). In various parts of the proof we will show that for $p> n/2+1$, the
$p$-Sylow subgroup of $G$ is either trivial or cyclic (which does not hold for $p\leq n/2+1$). 

\bigskip\noindent\textbf{Proof of Theorem \ref{P1_bis}}
Since we have already proved the statement for $n\in \{2,3\}$, we assume $n\geq 4$. If $p>n/2+1$, then in particular $p>3$. 
Suppose that  for every integer $n' <n$, the local-global divisibility by $p$ holds  in commutative algebraic groups $\A$ over $k$ when $G\leq \GL_{n'}(q)$ 
acts very strongly irreducibly on $\A[p]\simeq (\ZZ/p\ZZ)^{n'}$ and $p>\ddfrac{n'}{2}+1$. 
 As in the cases when $n=2$ and $n=3$, since $H^1_{\textrm{loc}}(G,.)$ is an additive functor, we may assume, without loss of generality, that 
$G$ is not of class $\C_1$.

\vskip 0.5cm
\noindent \emph{Part i. Subgroups of geometric type}

\bigskip \noindent \emph{Class $\C_2$}

\bigskip\noindent

\par Assume that $G$ is of class $\C_2$.  In this case $\A[p] =\bigoplus_{i=1}^{r} A_i$, with each $A_i$
of dimension $t$ and $G=(G_1\times ... \times G_r)\wr S_r$, where $G_i$ is a subgroup of $\GL_t(q)$ acting on $A_i$ and $tr=n$.
In this situation $\A[p]$ is irreducible, but not very strongly irreducible (nor strongly irreducible). By our assumptions, then $\A[p]$ 
has to be a direct  sum of very strongly irreducible $G$-modules. In particular $A_i$ is a very strongly irreducible $G_i$-module for all $1\leq i\leq r$.
Observe that the greatest $r$ that we can have is $r=n$ itself. But in this case $t=1$ and $\A[p] =\bigoplus_{i=1}^{n} A_i$, 
with each $A_i$ of dimension 1.  The group $G$ is then a subgroup of $C_{q-1}^n.S_n$. 
If $p>n/2+1$, then $2p>n+2$ and $p^2\nmid n!$. Thus the $p$-Sylow subgroup $G_p$ of $G$ is either trivial
or cyclic and $H^1_{\loc}(G,\A[p])=0$ by Lemma \ref{Sylow}. 
Let $r< n$. Observe that in particular we have
$r\leq n/2$. If $p>r$, then the $p$-Sylow subgroup of $G$ is contained in $G_1\times ... \times G_r$. Being $r\leq n/2$, then $p>r$ implies $p>n/2+1$. Observe that when $r<n$, then  $t< n$ too (in fact we have $t\leq n/2$). 
By induction,  for all $p>t/2+1$, we have $H^1_{\loc}(G_i, A_i)=0$, for every $1\leq i\leq r$.
Since $H^1_{\loc}(G, .)$ is an additive functor, then by induction we get that $H^1_{\loc}(G, \A[p])=0$, for $p> \max\left\{r,\ddfrac{t}{2}+1\right\}> \ddfrac{n}{2}+1$. 

\bigskip \noindent \emph{Class $\C_3$}

\bigskip\noindent

\par Suppose that  $G$ is of type $\C_3$. 
In this situation the $G$-module $\A[p]$ is considered as a vector space over a
a field extension $\widetilde{F}$ of $\FF_{p}$ with degree $r$, where
$r$ is a prime number dividing $n$. The $p$-torsion subgroup $\A[p]$
has dimension $t:=n/r$ as a vector space over $\widetilde{F}$ (see \cite[\S 5.3 and Table 2.1.A]{KL}) and
$G$ is isomorphic to a subgroup of $\GL_t(p^{mr}).C_r$. 
If $r=n$ (this in particular happens if $n$ is a prime), then the only possible subgroup of class $C_3$ is
$\GL_1(p^{m n}). C_n$. Since the cardinality of $\GL_1(p^{mn})$ is $\frac{p^{m n}-1}{p-1}$, then, for every $p$,
 the $p$-Sylow subgroup of $G$ is either trivial or cyclic and $H^1_{\loc}(G,\A[p])=0$.  
If $r< n$, then 
 $1<r\leq n/2$. 
Since $1< t \leq n/2$ too,  we use induction (recall that here
$\A[p]$ is considered as a vector space of dimension $t<n$ over $\widetilde{F}$) and Lemma \ref{extension} together with
Remark \ref{extension_rem} to get $H^1_{\textrm{loc}}(G,\A[p])=0$, for every $p>  n/4+1$ (i.e. in
particular for every $p>n/2+1$). 

\bigskip \noindent \emph{Class $\C_4$}

\bigskip\noindent

\par Suppose that $G$ is of type $\C_4$. 
The group $G$ is isomorphic to a subgroup of a central product $\GL_t(p^m)\circ \GL_r(p^m)$
 acting on a tensor product $V_1\otimes V_2=\A[p]$, where
 $rt=n$ and $V_1$, $V_2$ are vector spaces
 over $\F_{p}$, with dimension respectively $t$ and $r$. We are in the situation described in Lemma \ref{tensor}.
Thus let $G=G_t\circ G_r$, with $G_t$ acting on $V_1$ and $G_r$ acting on $V_2$  (see also \cite[\S 4.4]{KL}).
 Observe that $V_1$ itself is a very strongly irreducible $G_r$-module (in fact if $V_1$ has a proper
$G_r$-submodule, then $\A[p]$ has a proper $G_r$-submodule too). 
Since $r<n$, by induction  $H^{1}_{\textrm{loc}}(G_r,V_1)=0$, for every $p> r/2+1$. 
In the same way, since $t<n$, then by induction, for all $p> t/2+1$, 
we have $H^{1}_{\textrm{loc}}(G_t,V_2)=0$ (again $V_2$ is a  very strongly irreducible $G_t$-module).  
Observe that $r\leq  n/2 $ and $t\leq  n/2$. 
Thus $H^{1}_{\textrm{loc}}(G_1,\A[p])=0$ and  $H^{1}_{\textrm{loc}}(G_2,\A[p])=0$, for every $p> n/4+1$ (in particular for every $p>n/2+1$), 
implying $H^1_{\loc}(G,\A[p])=0$, by Lemma \ref{tensor}.

\bigskip \noindent \emph{Class $\C_5$}

\bigskip\noindent

\par If $G$ is of class $\C_5$, then $G$ is isomorphic to a subgroup of $\GL_n(p^t)$, where $m=tr$,
with $t$ a positive integer and $r$ a prime. Observe that this case does not occur when $m=1$.
If $G$ is the whole group $\GL_n(p^t)$, then by Remark \ref{SL_n} and by
Lemma \ref{whole}, we have $H^1_{\textrm{loc}}(G,\A[p])=0$. If $G$ is trivial, then  $H^1_{\textrm{loc}}(G,\A[p])$
is trivial too. Suppose that $G$ is a proper non-trivial subgroup of $\GL_n(p^t)$.
If $G$ is still of class $\C_5$, then $G$ is isomorphic to
a subgroup of $\GL_n(p^{t_2})$, for some positive integer $t_2$, such that $t=r_2t_2$, with $r_2$ prime.
If $G=\GL_n(p^{t_2})$, we have again $H^1_{\textrm{loc}}(G,\A[p])=0$, by Remark \ref{SL_n} and
Lemma \ref{whole}. Then we may
assume that $G$ is a proper subgroup of $\GL_n(p^{t_2})$ and so on.
Since $m$ is finite and we are assuming that $G$ is not trivial, then $G$
is isomorphic to a subgroup of $\GL_n(p^{t_j})$ (for some positive integer $t_j$ dividing $m$) of class
$\C_i$, with $i\neq 5$.  We may then repeat the arguments used (or that we will use) for other classes $\C_i$, with $i\notin \{1, 5\}$, to get
$H^{1}_{\textrm{loc}}(G,\A[p])=0$.

\bigskip \noindent \emph{Class $\C_6$}

\bigskip\noindent

\par Suppose that $G$ is of class $\C_6$, i.e. $G$ lies in the normalizer of an extraspecial group.
This may happen only when $n=r^t$, with $r$ a prime different from $p$ and $t$ a positive integer.
We have that $G$ is a subgroup of   $(C_{q-1}\circ r^{1+2t}).\Sp_{2t}(r)$  (see \cite[\S 3.5]{KL}).
To ease notation we denote by $H$ the normal subgroup $C_{q-1}\circ r^{1+2t}$.
Let $G':=G\cap H$. Owing to $p\neq r$, the $p$-Sylow subgroup $G_p'$ of $G'$ is trivial.
Then the $p$-Sylow subgroup $G_p$ of $G$ is isomorphic to the $p$-Sylow subgroup of $G/G'$ that is isomorphic to
a subgroup of $\Sp_{2t}(r)$. Since $p\neq r$, then $p$ divides the cardinality
$|\Sp_{2t}(r)|$ if and only if $p$ divides $\prod_{i=1}^{t}(r^{2i}-1)=(r-1)(r+1)...(r^{t-1}-1)(r^{t-1}+1)(r^{t}-1)(r^{t}+1)$.
Observe that $r^{t-1}=r^{t}/r\leq n/2$. If $p\neq 2$, then  $r^{t-1}=n/r< n/2$
and $r^{t-1}+1< n/2+1$. If $p>n/2+1$, then $p$ divides no factor in the product
$\prod_{i=1}^{t-1}(r^{2i}-1)$. Thus $p$ could divide only the factors $r^{t}-1$ and $r^{t}+1$.
If $p\mid r^{t}-1$ and $p\mid r^{t}+1$, then $p|2r^{t}$ and this is a contraddiction
for every odd primes $p\neq r$. 
So for all $p\neq 2$, if $p|(r^i+1)$, then $p\nmid (r^i-1)$ and the other way around.
The greatest factor of $(r-1)(r+1)...(r^{t-1}-1)(r^{t-1}+1)(r^{t}-1)(r^{t}+1)$ is
$r^t+1=n+1$. If $p> n/2+1$, then $p^2> n^2/4+n+1>n+1$, for all $n$. Thus
 if $p>n/2+1$, we have that $p^2\nmid |\Sp_{2t}(r)|$ and the $p$-Sylow subgroup of
$G$ is either trivial or cyclic. Consequently $H^1_{\textrm{loc}}(G, \A[p])=0$.

\bigskip \noindent \emph{Class $\C_7$}

\bigskip\noindent

\par Suppose that $G$ is of class $\C_7$.  This case 
 occur only when $n=r^t$, where $r$ is a prime and $t>1$.  
The group $G$ is the stabilizer of a tensor product
decomposition $\bigotimes_{i=1}^{t}V_i$, with $n=r^t$, $t\geq 2$ and $\textrm{dim}(V_i)=r$, for every $1\leq i\leq t$.  Thus
$G$ is a subgroup of $\underbrace{ (\GL_r(q)\circ ... \circ \GL_r(q)) }_{t}.S_t $. The situation is similar to the one described
in Lemma \ref{tensor3}  (with
$j_0=0$, for every $1\leq i\leq t$), but here we can have more than one permutation of the $V_i's$. 
If $p>t$, then the $p$-Sylow subgroup of $G$
is contained in $G'=G\cap \underbrace{ (\GL_r(q)\circ ... \circ \GL_r(q)) }_{t}$. Observe that $G'=G_1\circ ... \circ G_t$, where $G_i\leq GL_r(t)$ 
is a normal subgroup
of $G$,  acting on $V_i$, for every $1\leq i\leq t$. Being $\A[p]$ a very strongly irreducible $G$-module, we have that
$\A[p]$ is also a very strongly irreducible $G_i$-module, for every $1\leq i\leq t$. 
Furthermore $V_i$ is a very strongly irreducible $G_i$-module (in fact if $V_i$ has a proper
$G_i$-submodule, then $\A[p]$ has a proper $G_i$-submodule).  
Observe that $G'$ acts on $\A[p]$ as
described in Corollary \ref{tensor_2}.  Since $G_i$ acts on $V_i$, then
we  use induction on $r$ and we have $H^1_{\textrm{loc}}(G_i,V_i)=0$, for all $p>\ddfrac{r}{2}+1$, for every $1\leq i\leq t$. 
By Corollary \ref{tensor_2} we have $H^1_{\textrm{loc}}(G',\A[p])=0$, for all $p>  \max\{t,\ddfrac{r}{2}+1\}$.  
Obviously $r<n$ and $t\leq n/2$, being $t=\log_r(n)$. Then if we take  $p> n/2+1$, we still have that the $p$-Sylow subgroup of $G$ is contained in $G'$
and $H^1_{\textrm{loc}}(G',\A[p])=0$. This implies $H^1_{\textrm{loc}}(G,\A[p])=0$, for every $p>n/2+1$, by Lemma \ref{Sylow}.

\bigskip \noindent \emph{Class $\C_8$}

\bigskip\noindent

\par Suppose that $G$ is of class $\C_8$. If $n$ is even, then $G$ is contained
either in the group $\Sp_n(p^m)$, or in a group $\textrm{O}_n^{\epsilon}(p^m)$,
for some $\epsilon\in\{+,-\}$, or in the group $U_n(p^{\frac{m}{2}})$, with $m$ even too.
If $n$ is odd, then $G$ is contained
either in $\textrm{O}_n(p^m)$, or in $U_n(p^{\frac{m}{2}})$ (with $m$ even).
Observe that for symplectic groups  $\Sp_n(p^m)$, we have $n=2t$, for some positive integer $t$
and then $n/2+1=t+1$. If $G$ is one of the whole mentioned classical groups, then we may apply Lemma \ref{whole}
to get $H^1_{\textrm{loc}}(G,\A[p])=0$. 
Assume that $G$ is strictly contained in one of those classical groups. 
Aschbacher's theorem holds for unitary, symplectic and orthogonal groups too and the maximal subgroups of those classical groups are still divided
in the same 9 classes (see \cite{KL}).  
From the classification of the maximal subgroups of $\Sp_n(p^m)$, $\textrm{O}_n(p^m)$, $\textrm{O}_n^{\epsilon}(p^m)$ and $U_n(p^{\frac{m}{2}})$
of class $\C_i$, $i\neq 9$  appearing in \cite[Table 3.5B, Table 3.5C, Table 3.5D and Table 3.5E]{KL}, we have that $\textrm{O}_n(p^m)$, 
$\textrm{O}_n^{\epsilon}(p^m)$ and $U_n(p^{\frac{m}{2}})$ do not contain groups of class $\C_8$ and
that the subgroups of $\Sp_n(p^m)$ of class $\C_8$ are $\textrm{O}_n^{\epsilon}(p^m)$ themselves.
Since we are assuming that $G$ is strictly contained in one of those three groups, then it is a subgroup of class $\C_i$, for
some $i\neq 8$.  By repeating the arguments used for the maximal subgroups of $\SL_n(q)$ of class $\C_i$, with $i\neq 8$ (see Part \emph{ii} below for class $\C_9$),
for the maximal subgroups of symplectic, orthogonal and unitary groups, we get the conclusion.

\bigskip
\noindent \emph{Part ii. Subgroups of class $\C_9$}

\bigskip
\par Suppose that $G$ is of class $\C_9$ and let $Z(G)$ be its center. By the description of the subgroups of class $C_9$, recalled in Section \ref{sec1}, the group $G/Z(G)$ is almost simple. Assume that the $p$-Sylow subgroup $Z(G)_p$ of $Z(G)$ is nontrivial.
We recall that one of the $p$-Sylow subgroups of $\GL_n(p^m)$ is the subgroup ${\mathfrak{U}}_n$ of the upper triangular matrices with
all entries on the principal diagonal equal to 1. Let $M_1\in Z(G)_p$.  All the $p$-Sylow subgroups of $G$
are conjugate, 
then there exist $g\in \GL_n(p^l)$  and $M_2 \in{\mathfrak{U}}_n$ such that $M_1=gM_2 g^{-1}$.
This is equivalent to $gM_1g^{-1}=M_2$ and since $M_1$ commutes with all $g\in G$, then
$M_1=M_2$. Thus $M_1$ is in upper triangular form too. Now let $h\in G$. Again, since $M_1$
is in the center of $G$, then $hM_1=M_1h$. This implies that $h$ itself should be in triangular
form, which is a contraddiction with our assumption that $\A[p]$ is an irreducible $G$-module. Therefore $Z(G)_p$ is
trivial implying $H^1_{\loc}(Z(G),\A[p])=0$.  It is well-known that the restriction map

$$H^1(Z(G),\A[p])\rightarrow H^1(Z(G)_p,\A[p])$$

\noindent is injective on the $p$-primary part of $H^1(Z(G),\A[p])$ (see for example \cite[Theorem 4, Chap. IX, \S 2]{Ser3}).
Since $\A[p]\simeq (\ZZ/p\ZZ)^n$ is a $p$-group, the $p$-primary part of $H^1(Z(G),\A[p])$ is the whole group.
Then $H^1(Z(G)_p,\A[p])=0$ implies $H^1(Z(G),\A[p])=0$. By the
inflation-restriction exact sequence

$$ 0\rightarrow H^1(G/Z(G),\A[p]^{Z(G)})\rightarrow H^1(G,\A[p])\rightarrow H^1(Z(G),\A[p])^{G/Z(G)}, $$

\noindent we have $H^1(G/Z(G),\A[p]^{Z(G)})\simeq H^1(G,\A[p])$ and, in particular
 $H_{\loc}^1(G/Z(G),\A[p]^{Z(G)})\simeq H_{\loc}^1(G,\A[p])$.  
So it suffices to prove $H^1_{\loc}(G/Z(G),\A[p]^{Z(G)})=0$.
Since $Z(G)$ is a normal subgroup of $G$, with the same argument used in the
proof of Lemma \ref{whole}, we have that $\A[p]^{Z(G)}$ is a $G$-submodule
of $\A[p]$.
Being $\A[p]$ is an irreducible $G$-module, we deduce that $\A[p]^{Z(G)}$
is either trivial or it is the whole $\A[p]$. In the first case obviously
 $H_{\loc}^1(G/Z(G),\A[p]^{Z(G)})=0$. If $\A[p]^{Z(G)}=\A[p]$, then
we have to prove that  $H_{\loc}^1(\widetilde{G},\A[p])=0$, where $\widetilde{G}=G/Z(G)$
is an almost simple group. To ease notation, we will prove $H^1_{\loc}(G,\A[p])=0$,
whenever $G$ is an almost simple group.
In this case $G$ contains a simple group $S$ and it is contained in the automorphism group of $S$
\begin{equation} \label{almost_simple} S\leq G\leq \Aut(S). \end{equation} In particular, if $G$ is a subgroup of $\GL_n(p^m)$, then $S$ is a subgroup of $\GL_n(p^m)$ too.
So we will consider the possible simple groups $S$ contained in $\GL_n(p^m)$, for a certain $n$. Observe that $S$ is not a cyclic group by
\eqref{almost_simple}.
\noindent The classification of the finite simple groups is well-known, as well as the list of their automorphisms
groups. One of the most complete references in the literature is Wilson's book on finite simple groups \cite{Wil}. Following that text, 
we will divide the simple groups in four classes: alternating groups, sporadic groups,  classical groups and exceptional groups.
We will consider the twisted exceptional group (i.e. the Ree groups,
the Suzuki groups, the group $^3D_4(q)$ and the group $^2E_6(q)$) among the
exceptional groups.
We will also call groups of Lie type the classical and the exceptional groups.  \par

\bigskip \noindent \emph{Alternating groups}

\bigskip\noindent Assume that $S$ is an alternating group $A_N$, for some positive integer $N$. Since
the cardinality of the outer automorphism group of $A_N$ divides 4 for all $N$, then we may assume without loss of generality that $G=S$. By \cite[Proposition 5.3.7 (i)]{KL},
we have that the minimal degree for a  representation of  $A_N$ in $\GL_n(p^m)$, for $n\geq 9$, is $N-2$, i.e. $n\geq N-2$. Then
$N\leq n+2$. Since $|A_{n+2}|=\ddfrac{(n+2)!}{2}$, if  $p> n/2+1$, then
$p^2\nmid |A_{n+2}|$. In  particular $p^2$ does not divide the cardinality of every possible
subgroup of $\GL_n(p^m)$ of type $C_9$ isomorphic to $A_N$, for every  positive integer $N$ (for $n\geq 9$). 
We have $H^1_{\loc}(G,\A[p])=0$, for all $p> n/2+1$, $n\geq 9$. We will 
treat the case when $n\leq 8$ in Part \emph{iii.} below, showing that for all $n$ the bound
$p> n/2+1$ is indeed sufficient for the triviality of  $H^1_{\loc}(G,\A[p])$.

\par\bigskip \noindent \emph{Sporadic groups}

\bigskip\noindent Assume that $S$ is a sporadic group. Then $p=13$ is the greatest prime number such
that $p^2$ could eventually divide its cardinality (this is the case of the Monster group). 
Furthermore, for every sporadic group, the outer automorphism group is either trivial or cyclic of order 2.
Then for $p> 13$, we have $H^1_{\loc}(G,\A[p])=0$.
 If $n\geq 24$,  the bound $p> n/2+1$ covers the case of sporadic groups too, as well as the cases of the groups of geometric type and of the alternating groups. 
We will see in Part \emph{iii.} below that the same bound $p> n/2+1$ assures the triviality of  $H^1_{\loc}(G,\A[p])$, when $G$ is a sporadic group, 
for all $n\leq 23$.
\par
\bigskip \noindent \emph{Groups of Lie type}

\bigskip\noindent Now assume that $S$ is neither
alternating, nor sporadic. If $p>3$, then the automorphisms of the field $\FF_{p^{m}}$, generated by the Frobenius map $\f: x\mapsto x^p$, are 
the only automorphisms of $S$, whose order can be divided by $p$ (see for instance \cite{Wil}).
The Frobenius automorphisms form a group of outer automorphisms
of $S$ isomorphic to $C_{m}$.  Then we have outer automorphisms of $S$ with order divided by $p$  if and only if $p\mid m$ and we
 may assume  $G\simeq S.C_m$. 
Being $\A[p]$ very strongly irreducible, the vanishing of
$H^1_{\textrm{loc}}(S,\A[p])$ implies the vanishing of $H^1_{\textrm{loc}}(G,\A[p])$,
by Lemma \ref{extension} and Remark \ref{extension_rem}. So it suffices to prove $H^1_{\textrm{loc}}(S,\A[p])=0$, for all
groups $S$ of Lie type, whenever $p> n/2+1$. We are going to consider two distinct situations: when
the characteristic of the field of definition of $S$ is different from $p$ (the so-called cross characteristic case in the literature) and when
the characteristic of the field of definition of $S$ is equal to $p$ (the so-called defining characteristic case).

\par
\vskip 1cm \noindent \emph{Cross characteristic case}

\bigskip\noindent
If the characteristic of the field of definition of $S$ is different from $p$, then 
we have an explicit lower bound for the degrees of the representations of $S$, as one can see 
in \cite[Table 5.3.A, pag. 188]{KL} (see also \cite{Land} and  \cite{Hoff}). For instance,
if $S$ is isomorphic to $\PSL_2(r^{\alpha})$,  for some odd prime $r$ and some positive integer $\alpha$,
then $n\geq \ddfrac{r^{\alpha}-1}{2}$  (by \cite[Table 5.3.A, pag. 188]{KL}).
In this case $r^{\alpha}\leq 2n+1$. 
Being $r\neq p$,  if a prime  $p$ does not divide  $(r^{{\alpha}}+1)(r^{{\alpha}}-1)$, then $p$ does not
divide the cardinality 
of $\PSL_2(r^{\alpha})$. 
Observe that  
every odd prime $p\neq r$ that divides  $r^{{\alpha}}+1$ does not divide $r^{{\alpha}}-1$ and the other way around (as in the case of groups of class $\C_6$).
Moreover $r^{{\alpha}}+1\leq 2n+2$.
Suppose  $p> n/2+1$. Then $p^2> n^2/4+n+1$. We have $n^2/4+n+1> 2n+2$, for all $n> 4$.
Therefore the $p$-Sylow  subgroup of $S$ is either trivial or cyclic and $H^1_{\loc}(S,\A[p])=0$, for every $n>4$.
We will treat the case when $n=4$ in Part \emph{iii.} below.
\par  If $S$ is isomorphic to $\PSL_2(2^{\alpha})$, the bound for $n$ is $n\geq 2^{\alpha}-1$ (see again  \cite[Table 5.3.A, pag. 188]{KL}).
With analogous arguments as for odd primes $r$, if $p> n/2+1$, then $p^2$ does not divide $(2^{{\alpha}}+1)(2^{{\alpha}}-1)$, for all $n\geq 4$
and $H^1_{\loc}(G,\A[p])=0$. 
\par Now suppose  $S=\PSL_t(r^{\alpha})$, for some $t\leq n$ and $r\neq p$. 
The  bound 
for $n$ is $n\geq (r^{\alpha})^{t-1}-1$ \cite[Table 5.3.A, pag. 188]{KL}. Then $(r^{\alpha})^{t-1}\leq {n+1}$, implying
$t-1\leq \log_{r^{\alpha}}(n+1)$. Observe that $n-1> \log_{r^{\alpha}}(n+1)$, for all $n\geq 4$. Thus
$t\leq n-1$. On the other hand, the bound $n\geq (r^{\alpha})^{t-1}-1$ implies
$r^{\alpha}\leq \sqrt[t-1]{n+1}$.
As above, since $p\neq r$, if
$p\nmid \prod_{i=2}^{t}((r^{\alpha})^i-1)$,  then $p$ does not divide the cardinality of $\PSL_t(r^{\alpha})$.
Let $p> n/2+1$. Since $n\geq (r^{\alpha})^{t-1}-1$, we have $p> \ddfrac{n}{2}+1\geq \ddfrac{(r^{\alpha})^{t-1}-1}{2} +1$.
Observe that  $\ddfrac{(r^{\alpha})^{t-1}-1}{2}+1>(r^{\alpha})^i-1$, for all $i\leq t-2$. Therefore all $p> n/2+1$ do not divide each
factor in the product $\prod_{i=2}^{t-2}((r^{\alpha})^i-1)$. Consider the last two factors $(r^{\alpha})^{t-1}-1$ and $(r^{\alpha})^{t}-1$.
If $p| (r^{\alpha})^{t-1}-1$ and $p|(r^{\alpha})^{t}-1$, then $p|(r^{\alpha})^{t-1}(r^{\alpha}-1)$. 
Since $p\neq r$ and we have proved $p> r^{\alpha}-1$, then $p$ divides at most one factor
in the product $\prod_{i=2}^{t}((r^{\alpha})^i-1)$.
By $r^{\alpha}\leq \sqrt[t-1]{n+1}$, we have $(r^{\alpha})^{t-1}-1\leq \sqrt[t-1]{(n+1)^{t-1}}-1=n$.
If $p> n/2+1$, then $p^2> \ddfrac{n^2}{4}+n+1>n>\ddfrac{((r^{\alpha})^{t-1}-1)^2}{4}+(r^{\alpha})^{t-1}$. 
One can easily verify that $\ddfrac{((r^{\alpha})^{t-1}-1)^2}{4}+(r^{\alpha})^{t-1}\geq (r^{\alpha})^{t}-1$, for all $r$ when
$t>3$ and for all $r^{\alpha}\geq 4$ when $t=3$. In these cases $p^2\nmid | \PSL_t(r^{\alpha})|$ and
then the $p$-Sylow subgroup of $S$ is either trivial or cyclic, implying $H^1_{\loc}(S,\A[p])=0$.
We are left with the cases when $t=3$ and $r\leq 3$. If $r=2$, we have $\prod_{i=2}^{t}((r^{\alpha})^i-1)=
(2^2-1)(2^3-1)=3\cdot 7$. Again, for every $p\neq 2$, the $p$-Sylow subgroup of $S$  is either trivial or cyclic, implying $H^1_{\loc}(S,\A[p])=0$.
If $r=3$, then $\prod_{i=2}^{t}((r^{\alpha})^i-1)=
(3^2-1)(3^3-1)=2^4\cdot 13$. As above $H^1_{\loc}(S,\A[p])=0$, for all $p\neq 2$.
\par Now suppose $S=\PUni_t(r^{\alpha})$, for some $t\leq n$ and $r\neq p$. To ease notation, from now on let $r^{\alpha}=w$.
When $n$ is odd, the  bound is $n\geq w\ddfrac{w^{t-1}-1}{w+1}$,  and, when $n$ is even, the bound is
$n\geq \ddfrac{w^{t}-1}{w+1}$
 \cite[Table 5.3.A, pag. 188]{KL}. Suppose that $n$ is odd. 
Observe that  $\ddfrac{w}{w+1} \geq \ddfrac{2}{3}$, for every $w$.
 If a prime $p$ does not divide $\prod_{i=2}^{t}(w^i-(-1)^i)$, then it does not divides the cardinality of
 $S=\PUni_t(w)$. 
Let $p> n/2+1$.  Owing to  $n\geq w\ddfrac{w^{t-1}-1}{w+1}$ and $\ddfrac{w}{w+1} \geq \ddfrac{2}{3}$, we get
$p> \ddfrac{n}{2}+1\geq  \ddfrac{w}{2}\ddfrac{w^{t-1}-1}{w+1} +1\geq \ddfrac{w^{t-1}-1}{3} +1$.
Observe that $\ddfrac{w^{t-1}-1}{3} +1\geq w^{t-3}+1$, for all $w$ and all $t\geq 3$.
Thus, in particular   $\ddfrac{w^{t-1}-1}{3} +1\geq w^{t-3}-(-1)^{t-3}$, for all $w$ and all $t\geq 3$.  We consider separately the cases when $t\neq 2$
and $t=2$. Assume first that $t\neq 2$.
We have observed that $p$ does not divide the product $\prod_{i=2}^{t-3}(w^i-(-1)^i)$.
We are left with the factors $w^{i}-(-1)^{i}$, for $t-2\leq i\leq t$. Suppose that
$p$ divides both $w^{t-1}-(-1)^{t-1}$ and $w^{t}-(-1)^{t}$. 
Then $p$ divides
$w^{t-1}-(-1)^{t-1}+w^{t}-(-1)^{t}= w^{t-1}+w^{t}= w^{t-1}(1+w)$.
Being $p\neq r$, then $p\nmid w$ and $p| 1+w$. In particular $p\leq 1+w$, i.e. $w\geq p-1$.
Thus $p> \ddfrac{n}{2}+1\geq  \ddfrac{w}{2}\ddfrac{w^{t-1}-1}{w+1} +1\geq  \ddfrac{p-1}{2}\ddfrac{(p-1)^{t-1}-1}{p} +1$,
if and only if $p-1>  \ddfrac{p-1}{2}\ddfrac{(p-1)^{t-1}-1}{p}$, if and only if $\ddfrac{(p-1)^{t-1}-1}{2p}< 1$.
Since $t\neq 2$ this is a contradiction (recall $p>3$). In the same way one can prove that $p$ cannot
divide both $w^{t-2}-(-1)^{t-2}$ and $w^{t-1}-(-1)^{t-1}$. Suppose that  $p$ divides both $w^{t-2}-(-1)^{t-2}$ and $w^{t}-(-1)^{t}$. 
Then $p$ divides $w^{t}-(-1)^{t}-(w^{t-2}-(-1)^{t-2})= w^{t}-w^{t-2}=w^{t-2}(w^2-1)=w^{t-2}(w+1)(w-1)$.
Being $p\neq r$, then either $p$ divides $w+1$ or $p$ divides $w-1$. 
In both cases $p\leq 1+w$ and we have already shown that this is a contradiction.
Thus every $p>n/2+1$ cannot divide more than one factor in
the product  $\prod_{i=t-2}^{t}((w)^i-(-1)^i)$ and in particular $p$ could divide at most one of the factors with
$t-2\leq i\leq t$.  We are going to prove that $p^2$ divides no factor in that
product. 
We have $w^i-(-1)^i\leq w^t+1$, for each $t-2\leq i\leq t$. 
Since $n\geq w\ddfrac{w^{t-1}-1}{w+1}$, then $p> n/2+1$ implies
 $p^2> n^2/4+n+1\geq  \ddfrac{1}{4}\left(w\ddfrac{w^{t-1}-1}{w+1}\right)^2+ w\ddfrac{w^{t-1}-1}{w+1}+1$.
One can easily verifies that $ \ddfrac{1}{4}\left(w\ddfrac{w^{t-1}-1}{w+1}\right)^2+ w\ddfrac{w^{t-1}-1}{w+1}+1 \geq w^t+1$,
for all $w$, when $t\geq 4$ and for all $w\geq 7$, when $t=3$. In these cases
$p^2$ does not divide the cardinality of $\PUni_t(w)$ and the $p$-Sylow subgroup of $S$
is either trivial or cyclic, implying $H^1_{\loc}(S,\A[p])=0$. We are left
with the groups $\PUni_3(2)$, $\PUni_3(3)$, $\PUni_3(4)$ and $\PUni_3(5)$.
By computing their cardinalities, one easily verifies that the $p$-Sylow
subgroup of $S$ is trivial or cyclic, for all $p> 3$. Thus  $H^1_{\loc}(S,\A[p])=0$.
Assume that $t= 2$, then  $\prod_{i=2}^{t}((w)^i-(-1)^i)= w^2-1=(w-1)(w+1)$.
Since $p\neq r$ and $p\neq 2$, then $p$ divides at most one factor in the product. 
Moreover,  the bound for $n$ is $n\geq w\ddfrac{w-1}{w+1}$. Thus
$p>n/2+1$ implies  $p^2> \ddfrac{w^2}{4}\ddfrac{(w-1)^2}{(w+1)^2}+w\ddfrac{w-1}{w+1}+1$.
$\ddfrac{w^2}{4}\ddfrac{(w-1)^2}{(w+1)^2}+w\ddfrac{w-1}{w+1}+1\geq w+1$ if and only if
$\ddfrac{w}{4}\ddfrac{(w-1)^2}{(w+1)^2}+\ddfrac{w-1}{w+1}\geq 0$ and this is obviously
true for all $w$. So  $H^1_{\loc}(S,\A[p])=0$.
If $n$ is even the bound is $n\geq \ddfrac{w^{t}-1}{w+1}$.  With reasonings very much akin to the ones used for odd $n$, one can 
 get  $H^1_{\loc}(S,\A[p])=0$. \par  The other minimal bounds for
$n$ appearing in \cite[Table 5.3.A, pag. 188]{KL}, when $S$ is a classical group in cross characteristic, are very similar
to the ones already discussed. So, again with arguments  that are very much akin to the ones already shown when $S$ is the projective special linear group or the projective unitary group, 
one can verify that $H^1_{\loc}(S,\A[p])=0$, for all $p> n/2+1$, whenever $S$ is a classical group of Lie type in cross characteristic. 

\par Assume that $S$ is the exceptional group $E_8(w)$. In this case the lower bound for the dimension of the representation is $n\geq w^{27}(w^2-1)$ (again \cite[Table 5.3.A, pag. 188]{KL}). Thus $w^{27}\leq \ddfrac{n}{w^2-1}\leq \ddfrac{n}{3}$, i.e. $w\leq \sqrt[27]{\ddfrac{n}{3}}$. A prime $p\neq r$ divides the cardinality
of $E_8(w)$   if and only if it divides the product $\prod_{i=0}^{3}(w^{6i+2}-1)\prod_{i=2}^{5}(w^{6i}-1)$. We have
$\prod_{i=0}^{3}(w^{6i+2}-1)\prod_{i=2}^{5}(w^{6i}-1)$. 
Let $p> n/2+1$. Then $p>  \ddfrac{w^{27}(w^2-1)}{2}+1$. Observe that  $\ddfrac{w^{27}(w^2-1)}{2}+1\geq w^{24}-1$.
Then $p$ can divide only the greatest factor $w^{30}-1$ in the product $\prod_{i=0}^{3}(w^{6i+2}-1)\prod_{i=2}^{5}(w^{6i}-1)$. 
Since $w\leq \sqrt[27]{\ddfrac{n}{3}}$, we have
$w^{30}-1\leq  \sqrt[27]{(\ddfrac{n}{3})^{30}}-1$.
As above, if $p> n/2+1$, then $p^2> n^2/4+n+1$. We want to show that $\ddfrac{n^2}{4}+n+1 \geq \sqrt[27]{(\ddfrac{n}{3})^{30}}-1$, for all $n\geq 4$.
The inequality $\ddfrac{n^2}{4}+n+1 \geq \sqrt[27]{(\ddfrac{n}{3})^{30}}-1$ is equivalent to 
$(\ddfrac{n^2}{4}+n+2)^{27} \geq (\ddfrac{n}{3})^{30}$. 
Clearly $(\ddfrac{n^2}{4}+n+2)^{27} - (\ddfrac{n}{3})^{30}\geq 0$, for all $n\geq 4$.
Thus $p^2>  \sqrt[27]{(\ddfrac{n}{3})^{30}}-1 \geq w^{30}-1$. We have again that
 the $p$-Sylow subgroup of $S$ is either trivial or cyclic and
$H^1_{\loc}(S,\A[p])=0$.
In similar ways, using the bounds in \cite[Table 5.3.A, pag. 188]{KL}, one sees that the assumption $p> \ddfrac{n}{4}+1$
is always sufficient to get the conclusion  $H^1_{\loc}(S,\A[p])=0$, when  $S$ is an exceptional
group of Lie type in cross characteristic. 

\par
\bigskip \noindent \emph{Defining characteristic case}

\bigskip

\noindent Assume that the characteristic of the field of definition of $S$ is $p$. 
We have that $S$ is a classical group of Lie type (i.e. a projective classic matrix group), with dimension $t<  n$ (where $t=2h$, for some positive integer $h$,
when $S=\PSp_t(q)$ is a projective symplectic group or $S$ is a projective orthogonal group of type $+$ or $-$), 
or an exceptional group of Lie type  (see \cite[Table 5.4.C, pag. 200]{KL}). If $S$ is a classical group of Lie type
of dimension $t<  n$,  we have already showed in the proof of Lemma \ref{whole} that $H^1_{\loc}(S,\A[p])=0$, for all $p>n/2+1$.
Assume that $S$ is an exceptional group of Lie type. By \cite{CPS} (see in particular Table (4.5) at page 186) and \cite{CPS2} 
(see in particular Table (4.3) at page 193), the cohomology
group $H^1(S,\A[p])$ is trivial, for all $p>3$, when $n$ is the possible minimal degree of the representation of $S$
and when $n$ is the dimension of the Lie algebra with automorphism group $S$ (in this case $S$ has a natural
representation in dimension $n$ and  this often coincides with the representation of $S$ with minimal degree).
If the representation of $S$ is neither the minimal nor the natural one, then we can proceed as in the proof of Lemma \ref{whole}. 
Here we repeat the same arguments used in that proof for the reader's convenience, since some of the bounds for $p$ are different
for the exceptional groups.  
We first consider the groups $S$ that
are not twisted. Let
$L(\lambda)$ denote the irreducible $G$-module of highest weight $\lambda$.  
In \cite[Theorem 1.2.3]{Georgia} the authors prove that  the first cohomology group
$H^1(S,L(\lambda))$ is trivial, for all $p>3$, when $S$ is one of the groups $E_6(q)$ or
$G_2(q)$ and $\lambda$ is a fundamental dominant weight (or it is less than a fundamental dominant weight). 
In \cite[Theorem 1.2.3]{Georgia} the authors also prove that $H^1(S,L(\lambda))$ is trivial, for all $p>31$, when $S=E_8(q)$  
and $\lambda$ is a fundamental dominant weight (or it is less than a fundamental dominant weight).
If $p>n/2+1$ with $p<31$, then $n<2p-2<60$. Since the minimal degree for the representation of $E_8(q)$ is $248$ (see \cite[Table5.4.C, pag. 200]{KL}), then
 $H^1(S,L(\lambda))$ is trivial, for all $p>n/2+1$. If $S$ is not one of the groups $E_6(q)$, $G_2(q)$ and $E_8(q)$ in  \cite[Theorem 1.2.3]{Georgia}
there are other bounds for $p$ (similar to $p>31$ when $S=E_8(q)$) to have the triviality of $H^1(S,L(\lambda))$,
with $\lambda$ a a fundamental dominant weight (or less than a fundamental dominant weight). 
With very easy computations, akin to that for $S=E_8(q)$, one sees that
$H^1(S,L(\lambda))$ is trivial, for all $p>n/2+1$, when $S$ is every possible exceptional group of Lie type,
not twisted, and $\lambda$ is a fundamental dominant weight (or it is less than a fundamental dominant weight). 
 As mentioned above,
since we are assuming that $\A[p]$
is irreducible, then $\A[p]=L(\lambda)$, for
some dominant weight $\lambda$ (see \cite{His} and \cite{Che}). Again we have a decomposition $L(\lambda)=\otimes_{i=1}^s L(\omega_i)$, where
$s$ is a positive integer and $\omega_i$ is a fundamental dominant weight, for every $1\leq i\leq s$.
Thus $H^1_{\loc}(S,\A[p])\simeq H^1_{\loc}(S, \otimes_{i=1}^s L(\omega_i))$, for certain fundamental
weights $\omega_i$ and the group $S$ acts on  a tensor product decomposition. 
In particular $S$ acts on
$\A[p]$ in the same way as the subgroups of class $\C_4$ (see \cite[\S 4.4]{KL}) or more 
generally the group $S$ acts on $\A[p]$ as prescribed by Steingberg's Tensor Product Theorem (see \cite{Ste}).
The same situation happened in a part of the proof of Lemma \ref{whole}. We have $S=\otimes_{i=1}^s S_i$,
with $S_i$ acting on $L(\omega_i)$, for every $i$.
Since the mentioned Theorem 1.2.3 in \cite{Georgia} assures the triviality of $H^1(S_i,L(\omega_i))$, for all $1\leq i\leq s$, 
 by applying Lemma \ref{tensor3} (and Remark \ref{tensor4}),  one deduces  the triviality  of $H^1_{\loc}(S,\A[p])$.
We have to prove the same conclusion for twisted groups of Lie type. If we assume that $p>3$,
then we have neither Suzuki groups nor Ree groups in the defining characteristic (see \cite{Wil} for further details).
We are left with groups $^2E_6(q)$ and $^3D_4(q)$. 
 The group $^2E_6(q)$ is a subgroup of $E_6(q^2)$ modulo scalars (see \cite[4.11]{Wil}).
We may apply Shapiro's Lemma (see for instance \cite[Theorem 4.19]{Neu} or \cite[Lemma 6.3.2 and Lemma 6.3.4]{Wei}) to get

 $$H^1\left(E_6(q^2),\textrm{Ind}_{^2E_6(q)}^{E_6(q^2)}\A[p]\right)\simeq H^1(^2E_6(q), \A[p]),$$

\noindent with $\textrm{Ind}_{^2E_6(q)}^{E_6(q^2)}$   denoting   the   induced    $G$-module
   $\bigoplus_{i=1}^s \sigma_i(\A[p])$,   where 
$s:=[E_6(q^2):^2E_6(q)]$ is the index of $^2E_6(q)$ in $E_6(q^2)$ (see \cite[Definition 4.18]{Neu}),
$\sigma_i$ varies in a system of left coset  representatives of $H$ in $G$
and $\sigma_i(\A[p])$ is isomorphic to $\A[p]$, for every $1\leq i\leq s$. 
Being $H^1(G,-)$ and additive functor for every group $G$, we get

$$\bigoplus_{i=1}^s H^1(E_6(q^2)  , \sigma_i(\A[p]))\simeq  H^1(^2E_6(q), \A[p]).$$

\noindent We have already proved that  $H^1_{\loc}(E_6(q^2), \A[p])=0$, under the assumption that $\A[p]$ is
irreducible. Therefore $H^1_{\loc}(E_6(q^2)  , \sigma_i(\A[p]))=0$, for all $1\leq i\leq s$, and 
$H^1_{\loc}(^2E_6(q), \A[p])=0$.
The group $^3D_4(q)$  is a subgroup of $\Omega_8^+(q^3)$. 
The cited results in \cite{CPS} and in \cite{Georgia} hold for $\Omega_8^+(q^3)$.
Then we may apply  all the arguments as above to get $H^1_{\loc}( \Omega_8^+(q^3), \A[p])=0$
and deduce $H^1_{\loc}( ^3D_4(q), \A[p])=0$, by Shapiro's Lemma.

\bigskip  

\vskip 1cm
\noindent \emph{Part iii. The cases when $4\leq n\leq 23$}
\bigskip

In accordance with Part \emph{i.}, the bound $p> n/2+1$ is sufficient, whenever $G$ is a subgroup of class $\C_i$, with $1\leq i\leq 8$.
In Part \emph{ii.} we showed that for all $n>4$ the bound $p> n/2+1$ is sufficient, whenever $G$ is a subgroup of class $\C_9$, which is
neither an alternating group,  nor a sporadic group. In addition,
If $G$ is an alternating group or a sporadic group, we have proved that the same bound works
respectively for all $n\geq  9$ and for all $n\geq 24$. Thus for $n=4$ we are going  to control case by case
that the bound $p>n/2+1$ is sufficient for groups of class $\C_9$; in addition for $5\leq n\leq 8$ and $5\leq n\leq 23$
we are going to control that the same bound works respectively for alternating groups and sporadic groups.

\begin{description}

\item[ ] \emph{ $n=4$ } \par

 Let $d:=\gcd(q-1,4)$.
 By the classification of the maximal subgroups of  $\SL_4(q)$ appearing in \cite[Table 8.9, pag. 381]{BHR}, we have the following 
maximal subgroups  of class $\C_9$.

\begin{description}

\item[(a)]   the group $A_7$, only if $q=p=2$;
    \item[(b)]   the group $C_d\circ C_2\hspace{0.1cm}^{\cdot}\PSL_2(7)$, for $q=p\equiv 1,2,4 \modn 7)$, $p\neq 2$;
        \item[(c)]  the group $C_d\circ C_2\hspace{0.1cm}^{\cdot}A_7$, for $q=p\equiv 1,2,4 \modn 7)$, $p\neq 2$;
            \item[(d)]   the group $C_d\circ C_2\hspace{0.1cm}^{\cdot}\textrm{U}_4(2)$, for $q=p\equiv 1 \modn 6)$.

\end{description}

In cases \textbf{a)}, \textbf{(b)}, \textbf{(c)} and \textbf{(d)}
the $p$-Sylow subgroups of $G$ is either trivial or cyclic for all $p\geq 3$. 
 Then we get the conclusion.

\item[ ] \emph{ $n=5$ } \par
We have to prove $H^1_{\loc}(G,\A[p])=0$, for all $p> 3$.
By the classification of the maximal subgroups of of $\SL_5(q)$  appearing in  \cite[Table 8.19, pag. 386]{BHR},
among the groups of class $\C_9$ there are no alternating groups and the only  sporadic group is the Mathieu group $M_{11}$ (only if $q=3$).
 Since $M_{11}$ has cardinality   $2^4 \cdot 3^2 \cdot 5 \cdot 11$, then 
 for every $p\geq 5$, the $p$-Sylow subgroup of $G$ is either trivial or cyclic.
So the bound $p>3$ works.

\item[ ] \emph{ $n=6$ } \par
We have to prove $H^1_{\loc}(G,\A[p])=0$, for all $p>3$.
The maximal subgroups of $\SL_6(q)$ of class $\C_9$ 
concerning the cases of alternating groups or sporadic groups are the ones appearing in the following list (see \cite[Table 8.25, pag. 389]{BHR})

\begin{description}

\item[(a)]   a group of type $\C_9$, the group $C_2\times C_3^{\cdot} A_6.C_2$; 
  \item[(b)]  a group of type $\C_9$, the group $C_2\times C_3^{\cdot} A_6$; 
\item[(c)]  a group of type $\C_9$, the group $C_6^{\cdot} A_6$; 
\item[(d)]  a group of type $\C_9$, the group $C_6^{\cdot} A_7$; 
\item[(e)] a group of type $\C_9$, the group $C_2^{\cdot} M_{12}$. 
\end{description}

Some of the cases in the list can occur only under certain conditions of $q$. Since the proof does not depend on
those conditions, we avoided to write them, to ease the notation.

 One easily verifies that in all cases
\textbf{(a)}, \textbf{(b)}, \textbf{(c)},  \textbf{(d)} and \textbf{(e)}, the $p$-Sylow subgroup of $G$ is either trivial or cyclic, for all $p>3$.
Thus  $H^1_{\textrm{loc}}(G,\A[p])=0$, for every $p\geq 5$.

\item[ ] \emph{ $n=7$ and $ n=8$ } \par
In accordance with \cite[Table 8.36, pag. 395]{BHR} and  \cite[Table 8.45, pag. 399]{BHR}
for all $q$ there are neither alternating groups nor sporadic groups that are maximal subgroups respectively of $\SL_7(q)$ and $\SL_8(q)$.

\item[ ] \emph{ $n=9$ } \par
In accordance with \cite[Table 8.55, pag. 406]{BHR} there are no maximal subgroups of $\SL_9(q)$ of class $\C_9$ 
corresponding to the case of sporadic groups.

\item[ ] \emph{ $n=10$ } \par
Let $d:={\textrm{gcd}}(q-1,10)$. The maximal subgroups of $\SL_{10}(q)$ of class $\C_9$ corresponding to the case
of sporadic groups are  (see \cite[Table 8.61, pag. 410]{BHR})

\begin{description}
  \item[(d)]  $C_d{\circ}C_2^{\cdot}M_{12}$  (where $M_{12}$ is the Mathieu group of order $2^6 \cdot 3^3 \cdot 5 \cdot 11$), for $q=p \equiv 3 \textrm{ (mod 8)}$;
\item[(e)] $C_d{\circ}C_2^{\cdot}M_{12}.C_2$,  for $q=p \equiv 1 \textrm{ (mod 8)}$; 
 \item[(f)]  $C_d{\circ}C_2^{\cdot}M_{22}$  (where $M_{22}$ is the Mathieu group of order $2^7 \cdot 3^2 \cdot 5\cdot 7 \cdot 11$),  for $q=p \equiv 11,15,23 \textrm{ (mod 28)}$;
\item[(g)] $C_d{\circ}C_2^{\cdot}M_{22}.C_2$, for $q=p \equiv 1,9,25 \textrm{ (mod 28)}$; 
\end{description}

\noindent Observe that $d| 10$, but $d\neq 5$, when $p=5$. Then,  in  all the cases \textbf{(a)},  \textbf{(b)},    \textbf{(c)} and \textbf{(d)}, for all $p\geq 5$, 
the $p$-Sylow subgroup $G_p$ of $G$ is isomorphic to a subgroup of one of the groups
 $M_{12}$ and $M_{22}$, that have cardinalities respectively $2^6 \cdot 3^3 \cdot 5 \cdot 11$ and $2^7 \cdot 3^2 \cdot 5 \cdot 7 \cdot 11$. 
Thus $G_p$ is either trivial or cyclic, for all $p\geq 5$.  In particular we have $H^1_{\textrm{loc}}(G,\A[p])=0$, for all $p\geq 7$, as required.
\smallskip

\item[ ] \emph{ $11\leq n\leq 23$ } \par
In \cite{HM},  the authors list all the possible subgroups of class $\C_9$ of $\SL_n(q)$,
for every $n\leq 250$, excluding the groups of Lie type in the defining characteristic (see also \cite{HM2}). 
In particular they show all possible  maximal subgroups of $\SL_n(q)$ of class  $\C_9$ corresponding to the case of sporadic groups.
Proceeding as for $n\leq 10$, by analyzing the tables in \cite{HM}, one sees that even when $11\leq n\leq 23$ the
first local cohomology group $H^1_{\textrm{loc}}(G,\A[p])$ is trivial for all $p> n/2+1$.

\end{description}

\par We can conclude, that for all $n\geq 2$, when $p>\ddfrac{n}{2}+1$ the local-global divisibility
by $p$ holds in $\A$ over $k$.   \eop



\bigskip  About the bound for $p$ we can make the following considerations.

\begin{rem}  \label{rem_sharp} \textbf{A likely sharp bound.}
 Looking at our proofs for $n\in\{2,3\}$ and $n\geq 4$ and looking at the tables listing all the maximal subgroups
of $\SL_n(q)$ for all $4\leq n\leq 12$ in \cite{BHR} and the tables listing all the maximal subgroups
of $\SL_n(q)$ except the groups of Lie type in the defining characteristic in \cite{HM}, one sees that the bound $p> n/2+1$
is probably sharp in many cases. In fact,  for $p\leq n/2+1$, the $p$-Sylow subgroup of $G$ could be  
a direct product of two cyclic groups $C_p$ (look for examples at the
$3$-Sylow subgroup of  of $A_7$, when $n=4$,  or at the $3$-Sylow subgroup $M_{11}$ when $n=5$, and so on).
When the $p$-Sylow subgroup $G_p$ of $G$ is isomorphic to $C_p^2$, the
local-global divisibility may fail as in the mentioned examples produced in \cite{DZ}, \cite{DZ3} and in \cite{Pal}, \cite{Pal2}, \cite{Pal3}.
In principle one could rise those examples to similar ones for all $n$. 
So the local-global principle for divisibility by $p$ could
fail for $p\leq n/2+1$, when $G_p\simeq C_p^2$. It is also interesting that the bound $p> n/2+1$ is what we need
to apply \cite[Theorem 1.2.3]{Georgia} in the proof of Lemma  \ref{whole}. In fact (as already stated for groups of class $\C_8$ that are
symplectic groups), when $n=2t$, then $n/2+1$ is exactly equal to the $t+1$ appearing in that proof. 
This gives an evidence that the bound $n/2+1$ should be necessary in
many cases. In fact that bound appears independently both in the part of the proof of our Theorem \ref{Pal17}
concerning Lie groups in cross characteristic and in the mentioned Theorem 1.2.3 in \cite{Georgia}.
Anyway for some $n$, the bound $p> n/2+1$ is not sharp; for instance 
for $n=10$, with a bit of computation one can see that the bound $p\geq 5$ is sufficient to get the conclusion
of Theorem \ref{Pal17}. In those cases $G_p$ is trivial or cyclic even for some
$p\leq  n/2+1$.   We will now replace the bound $p>n/2+1$ with a bound $p>p_n$, that is probably sharp, where $p_n$ is a prime less or equal than $n/2+1$, depending
only on $n$, for every $n$. Indeed, we will have that the $p_n$-Sylow subgroup of $G$ can be  isomorphic to $C_{p_n}^2$ and the local-global principle for divisibility by $p_n$ can fail. We cannot prove that the bound $p_n$
is really sharp, only because we cannot prove that  there exists a commutative algebraic group
$\A$ with a prescribed $p_n$-torsion subgroup $\A[p_n]$ such that $G_{p_n}$ is exactly a group  for which the
principle fails. We can only deduce from the proof of Theorem \ref{Pal17} that the group $G_{p_n}$ \emph{could} be isomorphic to $C_{p_n}^2$ and that 
this surely does not happen when $p> p_n$. By eventually changing the field of definition $k$, it is likely
that we can have $G_{p_n}\simeq C_{p_n}^2$. This still does not assure that $H^1_{\loc}(G_{p_n}, \A[p_n])=0$.
But among so many commutative algebraic groups $\A$ and number fields $k$, for each $n$, 
we expect that this happens for at least one of them, as in the case when $n=2$ for elliptic curves. 
We are going to give a new version of Theorem
\ref{Pal17} with such a bound $p_n$.

\par
For every $n$, let $\rho_n$ be the smallest prime such that, for all $p>  \rho_n$ the square $p^2$ 
divides no cardinalities of the maximal subgroups of class $\C_9$ of $\GL_n(q)$, excluding the groups of Lie type in the defining characteristic. 
Moreover let $2t_n$ be the greatest dimension of a projective symplectic group which is a subgroup of class $\C_9$ of $\SL_n(p^m)$ in the
defining characteristic case.
In addition, when $n=r^t$, for some prime $r$ and some positive integer $t$, let  ${\mathfrak{p}}_n$ be the smallest prime such that
for all $p> {\mathfrak{p}}_n$, the square $p^2$ does not divide $\prod_{i=1}^{t}(r^{2i}-1)$.
Observe that then $p^2$ divides no cardinalities of the subgroups of class $\C_6$ of $\GL_n(q)$.  
If $n$ is not a power of a prime, there are no subgroups of class $\C_6$ in $\GL_n(q)$, so set ${\mathfrak{p}}_n=1$ in that case. 
It is then clear from the proof of Theorem \ref{P1_bis}, that we can give a new version of Theorem
\ref{Pal17} as follows.

\begin{thm} \label{sharp}
Let $p$ be a prime number. Let $k$ be a number field and let $\A$ be a commutative algebraic group defined over $k$, with $\A[p]\simeq (\ZZ/p\ZZ)^n$. 
For every $n$, there exists a prime $p_n$, depending only on $n$, such that if $p> p_n$ and
$\A[p]$ is a very strongly irreducible $G_k$-module or a direct sum of very strongly irreducible $G_k$-modules, then
 the local-global divisibility by $p$ holds in $\A$ over $k$ and $\Sha(k,\A[p])=0$. Moreover $p_n=\max\left\{p_d,t_n+1,{\mathfrak{p}}_n, \rho_n\right\}$, where $d$ is the greatest divisor of $n$.
\end{thm}

\end{rem}

\bigskip  We now proceed with the proofs of the corollaries stated in Section \ref{sec0}, that can be quickly deduced from the proof
of Theorem \ref{P1_bis}.

\bigskip \noindent \textbf{Proof of Corollary \ref{Nori2}}
By the proof of Theorem \ref{P1_bis} Part \emph{i.}, concerning subgroups of class $\C_6$, 
one easily deduces that for $p> n+1$, the $p$-Sylow subgroup $G_p$ of $G$ is trivial. Thus $H^1(G_p,\A[p])=0$.   It is well-known that the restriction map

$$H^1(G,\A[p])\rightarrow H^1(G_p,\A[p])$$

\noindent is injective on the $p$-primary part of $H^1(G,\A[p])$ (see for example \cite[Theorem 4, Chap. IX, \S 2]{Ser3}).
Since $\A[p]\simeq (\ZZ/p\ZZ)^n$ is a $p$-group, the $p$-primary part of $H^1(G,\A[p])$ is the whole group.
Then $H^1(G_p,\A[p])=0$ implies $H^1(G,\A[p])=0$. \eop

\bigskip \noindent \textbf{Proof of Corollary \ref{Nori}}
In the proof of Theorem \ref{P1_bis}  Part \emph{ii.}, we showed that if $p>n/2+1$, then $p^2$ does
not divide the cardinality of any subgroup of $\GL_n(p^m)$, which is a group of Lie type in cross characteristic,
except a few cases when $n=3$ and $r^{\alpha}\leq 5$. If $n=3$ and $r^{\alpha} \leq 5$, then
every $p>63$ does not divide the cardinality of any group of this type.
So, if $p>\max\left\{63,(\ddfrac{n}{2}+1)^2\right\}$, then $p$ does not divide  the cardinality of any subgroup of $\GL_n(p^m)$, which is a group of Lie type in cross characteristic
and $H^1(G_p,\A[p])=0$.   
As in the proof of Corollary \ref{Nori2}, this implies $H^1(G,\A[p])=0$. \eop

\begin{rem} In the same way as Corollary \ref{Nori} and Corollary \ref{Nori2} one sees that 
\begin{description} 

\item[a)] if $p> n+2$ and  the  absolute Galois group $G_k$ acts on $\A[p]$ as a subgroup  of   an alternating group, then  
$H^1(G, \A[p])=0$.

\item[b)] if $p> 13$ and the  absolute Galois group $G_k$ acts on $\A[p]$ as a subgroup  of a sporadic group, then  
$H^1(G, \A[p])=0$.

\end{description}
\end{rem}

\bigskip\noindent \emph{Acknowledgments}.  
I am grateful to John van Bon,  Gabriele Ranieri and Jacob Stix for useful discussions.
I wrote a part of this paper at the Max Planck Institute for Mathematics in Bonn. I would like to warmly 
thank all people there for their kind hospitality and for the excellent work conditions.  
I am also grateful to Istituto Nazionale di Alta Matematica ``F. Severi'' that partially supported 
this research with grant ``Assegno di ricerca Ing. Giorgio Schirillo''. 
Finally, I would like to thank the anonymous referee for his precious and constructive remarks, that allow me to
improve the paper and shorten the proofs of Lemma \ref{extension} and its corollaries.

\vskip 0.5cm

\newpage

Laura Paladino\par\smallskip
Department of Mathematics\par
University of Pisa \par
Largo Bruno Pontecorvo, 5\par
56127 Pisa\par
Italy\par 
e-mail addresses: laura.paladino@dm.unipi.it,  paladino@mat.unical.it


\end{document}